\documentclass[11pt]{amsart}
\usepackage{amsfonts,latexsym,amsthm,amssymb,graphicx}
\usepackage[all]{xy}
\usepackage[usenames]{color}
\usepackage{stackengine}
\usepackage{tikz}
\usetikzlibrary{intersections,calc,arrows.meta}
\usepackage[enableskew]{youngtab}
\usepackage{tableau,ytableau}
\usepackage[position=bottom]{subfig}

\allowdisplaybreaks

\usepackage{xypic,amscd,epsfig}
\usepackage{hyperref}
\usepackage[capitalize,nameinlink]{cleveref}

\DeclareFontFamily{OT1}{rsfs}{}
\DeclareFontShape{OT1}{rsfs}{n}{it}{<-> rsfs10}{}
\DeclareMathAlphabet{\mathscr}{OT1}{rsfs}{n}{it}

\CompileMatrices

\newtheorem{theorem}{Theorem}[section]
\newtheorem{lemma}[theorem]{Lemma}
\newtheorem{corol}[theorem]{Corollary}
\newtheorem{prop}[theorem]{Proposition}

\newtheorem{defin}[theorem]{Definition}
\newtheorem{remark}[theorem]{Remark}
\newtheorem{example}[theorem]{Example}

\newcommand{\nc}{\newcommand}
\nc{\rnc}{\renewcommand}
\nc{\bb}[1]{{\mathbb #1}}
\nc{\bbA}{\bb{A}}\nc{\bbB}{\bb{B}}\nc{\bbC}{\bb{C}}\nc{\bbD}{\bb{D}}
\nc{\bbE}{\bb{E}}\nc{\bbF}{\bb{F}}\nc{\bbG}{\bb{G}}\nc{\bbH}{\bb{H}}
\nc{\bbI}{\bb{I}}\nc{\bbJ}{\bb{J}}\nc{\bbK}{\bb{K}}\nc{\bbL}{\bb{L}}
\nc{\bbM}{\bb{M}}\nc{\bbN}{\bb{N}}\nc{\bbO}{\bb{O}}\nc{\bbP}{\bb{P}}
\nc{\bbQ}{\bb{Q}}\nc{\bbR}{\bb{R}}\nc{\bbS}{\bb{S}}\nc{\bbT}{\bb{T}}
\nc{\bbU}{\bb{U}}\nc{\bbV}{\bb{V}}\nc{\bbW}{\bb{W}}\nc{\bbX}{\bb{X}}
\nc{\bbY}{\bb{Y}}\nc{\bbZ}{\bb{Z}}
\nc{\mbf}[1]{{\mathbf #1}}
\nc{\bfA}{\mbf{A}}\nc{\bfB}{\mbf{B}}\nc{\bfC}{\mbf{C}}\nc{\bfD}{\mbf{D}}
\nc{\bfE}{\mbf{E}}\nc{\bfF}{\mbf{F}}\nc{\bfG}{\mbf{G}}\nc{\bfH}{\mbf{H}}
\nc{\bfI}{\mbf{I}}\nc{\bfJ}{\mbf{J}}\nc{\bfK}{\mbf{K}}\nc{\bfL}{\mbf{L}}
\nc{\bfM}{\mbf{M}}\nc{\bfN}{\mbf{N}}\nc{\bfO}{\mbf{O}}\nc{\bfP}{\mbf{P}}
\nc{\bfQ}{\mbf{Q}}\nc{\bfR}{\mbf{R}}\nc{\bfS}{\mbf{S}}\nc{\bfT}{\mbf{T}}
\nc{\bfU}{\mbf{U}}\nc{\bfV}{\mbf{V}}\nc{\bfW}{\mbf{W}}\nc{\bfX}{\mbf{X}}
\nc{\bfY}{\mbf{Y}}\nc{\bfZ}{\mbf{Z}}
\nc{\bfa}{\mbf{a}}\nc{\bfb}{\mbf{b}}\nc{\bfc}{\mbf{c}}\nc{\bfd}{\mbf{d}}
\nc{\bfe}{\mbf{e}}\nc{\bff}{\mbf{f}}\nc{\bfg}{\mbf{g}}\nc{\bfh}{\mbf{h}}
\nc{\bfi}{\mbf{i}}\nc{\bfj}{\mbf{j}}\nc{\bfk}{\mbf{k}}\nc{\bfl}{\mbf{l}}
\nc{\bfm}{\mbf{m}}\nc{\bfn}{\mbf{n}}\nc{\bfo}{\mbf{o}}\nc{\bfp}{\mbf{p}}
\nc{\bfq}{\mbf{q}}\nc{\bfr}{\mbf{r}}\nc{\bfs}{\mbf{s}}\nc{\bft}{\mbf{t}}
\nc{\bfu}{\mbf{u}}\nc{\bfv}{\mbf{v}}\nc{\bfw}{\mbf{w}}\nc{\bfx}{\mbf{x}}
\nc{\bfy}{\mbf{y}}\nc{\bfz}{\mbf{z}}

\nc{\mcal}[1]{{\mathcal #1}}
\nc{\calA}{\mcal{A}}\nc{\calB}{\mcal{B}}\nc{\calC}{\mcal{C}}\nc{\calD}{\mcal{D}}
\nc{\calE}{\mcal{E}} \nc{\calF}{\mcal{F}}\nc{\calG}{\mcal{G}}\nc{\calH}{\mcal{H}}
\nc{\calI}{\mcal{I}}\nc{\calJ}{\mcal{J}}\nc{\calK}{\mcal{K}}\nc{\calL}{\mcal{L}}
\nc{\calM}{\mcal{M}}\nc{\calN}{\mcal{N}}\nc{\calO}{\mcal{O}}\nc{\calP}{\mcal{P}}
\nc{\calQ}{\mcal{Q}}\nc{\calR}{\mcal{R}}\nc{\calS}{\mcal{S}}\nc{\calT}{\mcal{T}}
\nc{\calU}{\mcal{U}}\nc{\calV}{\mcal{V}}\nc{\calW}{\mcal{W}}\nc{\calX}{\mcal{X}}
\nc{\calY}{\mcal{Y}}\nc{\calZ}{\mcal{Z}}
\nc{\fA}{\frak{A}}\nc{\fB}{\frak{B}}\nc{\fC}{\frak{C}} \nc{\fD}{\frak{D}}
\nc{\fE}{\frak{E}}\nc{\fF}{\frak{F}}\nc{\fG}{\frak{G}}\nc{\fH}{\frak{H}}
\nc{\fI}{\frak{I}}\nc{\fJ}{\frak{J}}\nc{\fK}{\frak{K}}\nc{\fL}{\frak{L}}
\nc{\fM}{\frak{M}}\nc{\fN}{\frak{N}}\nc{\fO}{\frak{O}}\nc{\fP}{\frak{P}}
\nc{\fQ}{\frak{Q}}\nc{\fR}{\frak{R}}\nc{\fS}{\frak{S}}\nc{\fT}{\frak{T}}
\nc{\fU}{\frak{U}}\nc{\fV}{\frak{V}}\nc{\fW}{\frak{W}}\nc{\fX}{\frak{X}}
\nc{\fY}{\frak{Y}}\nc{\fZ}{\frak{Z}}
\nc{\fa}{\frak{a}}\nc{\fb}{\frak{b}}\nc{\fc}{\frak{c}} \nc{\fd}{\frak{d}}
\nc{\fe}{\frak{e}}\nc{\fFf}{\frak{f}}\nc{\fg}{\frak{g}}\nc{\fh}{\frak{h}}
\nc{\fri}{\frak{i}}\nc{\fj}{\frak{j}}\nc{\fk}{\frak{k}}\nc{\fl}{\frak{l}}
\nc{\fm}{\frak{m}}\nc{\fn}{\frak{n}}\nc{\fo}{\frak{o}}\nc{\fp}{\frak{p}}
\nc{\fq}{\frak{q}}\nc{\fr}{\frak{r}}\nc{\fs}{\frak{s}}\nc{\ft}{\frak{t}}
\nc{\fu}{\frak{u}}\nc{\fv}{\frak{v}}\nc{\fw}{\frak{w}}\nc{\fx}{\frak{x}}
\nc{\fy}{\frak{y}}\nc{\fz}{\frak{z}}

\newcommand{\C}{{\mathbb C}}

\newcommand{\bZ}{{\mathbb Z}}

\newcommand{\cF}{{\mathcal F}}

\newcommand{\Gr}{\mathrm{Gr}}

\newcommand{\csm}{c_{\mathrm{SM}}}
\newcommand{\ssm}{s_{\mathrm{M}}}

\newcommand{\one}{1\hskip-3.5pt1}

\newcommand{\h}{\mathrm{ht}}

\DeclareMathOperator{\Pic}{Pic}
\DeclareMathOperator{\Frac}{Frac}

\DeclareMathOperator{\Red}{Red}
\DeclareMathOperator{\STab}{STab}
\DeclareMathOperator{\Stab}{Stab}
\DeclareMathOperator{\loc}{loc}

\DeclareMathOperator{\mult}{mult}

\begin{document}
\title{Hook formulae from Segre--MacPherson classes}

\author{Leonardo C.~Mihalcea}
\address{
Department of Mathematics, 
Virginia Tech University, 
Blacksburg, VA 24061
USA
}
\email{lmihalce@vt.edu}

\author{Hiroshi Naruse}
\address{Graduate School of Education, University of Yamanashi, 
Kofu, 400-8510, Japan}
\email{hnaruse@yamanashi.ac.jp}

\author{Changjian Su}
\address{Department of Mathematics, University of Toronto, Toronto, ON M5S 2E4, Canada}
\email{csu@math.toronto.edu}


\date{\today}

\begin{abstract}
Nakada's colored hook formula is a vast generalization of many important formulae in combinatorics, such as the classical hook length formula and the Peterson's formula for the number of reduced expressions of minuscule Weyl group elements.~In this paper, we utilize cohomological properties of Segre--MacPherson classes of Schubert cells and varieties to prove a generalization of a cohomological version of Nakada's formula, in terms of smoothness properties of Schubert varieties. A key ingredient in the proof is the study of a decorated version of the Bruhat graph. Summing over weighted paths of this graph give the terms in the generalized Nakada's formula, and also provide algorithms to calculate structure constants of multiplications of Segre--MacPherson classes of Schubert cells. For simply laced Weyl groups, we also show the equality of `skew' and `straight' Nakada's formulae. This utilizes a criterion for smoothness in terms of excited diagrams of heaps of minuscule elements, which might be of independent interest.
\end{abstract}

\maketitle

\setcounter{tocdepth}{1}
\tableofcontents

\section{Introduction} The goal of this paper is to reprove and generalize a version of 
Nakada's colored hook formula, using localization properties of the Segre--MacPherson 
classes of Schubert varieties. We recall next the version of Nakada's formula utilized in this 
paper.

Let $G$ be a complex semisimple Lie group and fix opposite Borel groups $B,B^-$, giving  
$B \cap B^- = T$, a maximal torus in $G$. Let $R^+$ be the set of positive roots in $B$, and let $W:=N_G(T)/T$ be
the Weyl group, endowed with the Bruhat order $<$, and its length function $\ell: W \to \mathbb{N}$. For a Weyl 
group element $w \in W$, denote by $S(w) = \{ \beta \in R^+: s_\beta w < w \}$. 

Let $w\in W$ be a $\pi$-minuscule Weyl group element for an integral dominant weight $\pi$, in the sense of Peterson;
see \cite{proctor:bruhat,carrell:vector,proctor:minuscule,stembridge:fullycomm,stembridge2001minuscule,graham.kreiman:cominuscule_pt} and \S \ref{sec:minuscule} below. Nakada's colored hook formula \cite{nakada:colored} is the identity:
\begin{equation}\label{E:Nakada-intro}
\sum \frac{1}{\beta_1}\cdot \frac{1}{\beta_1+\beta_2}\cdot \ldots \cdot\frac{1}{\beta_1+\beta_2+\cdots +\beta_r}
=\prod_{\beta\in S(w)}\left(1+\frac{1}{\beta}\right) \/.
\end{equation}
Here the sum is over $r \ge 0$ and
oriented paths $x_r\overset{\beta_r}{\to} x_{r-1}\to \ldots \to x_1\overset{\beta_1}{\to} x_0=w$
in the Bruhat graph of the flag manifold $G/P$, where $P$ is the parabolic group with Weyl group $W_P = \Stab_W(\pi)$,
and the notation $x \overset{\beta}{\to} y$ means that $x,y $ are minimal length representatives such that 
$y W_P= s_\beta x W_P >xW_P$. See \S \ref{sec:hook} below for details. 

Nakada's formula is a vast generalization of several remarkable combinatorial formulae. One is the classical 
Frame--Robinson--Thrall hook formula calculating the dimension of the 
irreducible representation of the symmetric group $S_n$ indexed by the partition $\lambda$, or, equivalently, 
the number of standard tableaux of shape $\lambda$:
\begin{equation*}\label{equ:hook}
\chi_\lambda(1)=\#\STab(\lambda)= \frac{n!}{\prod_{\square\in \lambda}h_\square} \/.
\end{equation*}
Here $h_\square$ is the hook length of a cell $\square$ in the Young diagram of $\lambda$. 
Another special case is
the Peterson formula counting the number $\#\Red(w)$ of reduced decompositions of a 
$\pi$-minuscule Weyl group element $w$:
\begin{equation}\label{E:introredw} \#\Red(w)=\frac{\ell(w)!}{\prod_{\beta\in S(w)}\h(\beta)} \/. \end{equation}
We refer to \cite{nakada:colored} or to \S \ref{sec:col-and-conseq} below for more details about how to obtain these specializations from Nakada's formula. 

Let $W^P$ be the set of minimal representatives for the quotient $W/W_P$. 
For each $w \in W^P$, denote by 
$Y(v):=\overline{B^-w P/P} \subset G/P$ the corresponding Schubert variety.  
Our main goal is to prove an identity generalizing \Cref{E:Nakada-intro} in three directions:
\begin{enumerate}
\item We remove the hypothesis that $w \in W$ is (Peterson) $\pi$-minuscule.
\item We consider a `skew-version', for paths $v \le x_r \to \ldots \to x_0 = w$ where $v$ and $w$ are fixed, and such that 
the Schubert variety $Y(v)$ is smooth at $w$. Nakada's formula is obtained from specializing $v=id$, 
since $Y(id) = G/P$ is smooth at each $w$.
\item Instead of fixing a single integral weight $\pi$, we consider an {\bf admissible weight function}
$\Lambda: [v,w]^P \to {X^*(T)_P}$ from the Bruhat interval $[v,w]^P$ to the set of integral weights which stabilize $P$.
\end{enumerate}
The admissible functions are required to satisfy certain hypotheses spelled out in \S \ref{sec:LambdaBruhat}. 
An admissible function $\Lambda$ leads to a 
decorated version of the usual Bruhat graph which we call the {\bf $\Lambda$-Bruhat graph} of $G/P$.
This graph was utilized in \cite{mihalcea:eqqgr,mihalcea:eqqhom} as part of an algorithm calculating the Schubert structure constants of the equivariant quantum cohomology ring of $G/P$. 
Our most general result is the following (cf.~\Cref{thm:genNak}).
\begin{theorem}\label{thm:intro-main} Let $v\leq  w\in W^P$, and fix an admissible 
function $\Lambda:[v,w]^P\to X^{*}(T)_P$ with the associated 
$\Lambda$-Bruhat graph $\Gamma$. Set $S(w/v):=\{\beta\in R^+\mid v\le s_\beta w<w\}$.Then: 
\begin{center}
$Y(v)\subset G/P$ is smooth at $w\in G/P$  if and only if
\begin{equation}\label{E:intro-main}
\sum\frac{m_\Lambda(x_r,x_{r-1})}{\mathcal{W}_\Lambda(x_r)}\cdot \frac{m_\Lambda(x_{r-1},x_{r-2})}{\mathcal{W}_\Lambda(x_{r-1})}\cdot \ldots \cdot \frac{m_\Lambda(x_{1},x_{0})}{\mathcal{W}_\Lambda(x_{1})}
\;=\;
\prod_{\beta \in S(w/v)} \left(1+\frac{1}{\beta}\right),\end{equation}
\end{center}
where the sum is over integers $r \ge 0$, and over all directed paths 
$v\leq x_r \overset{}{\to} x_{r-1} \overset{}{\to} \ldots \overset{}{\to} x_0 = w$ in $\Gamma$.
Here $m_\Lambda(x,y) \in \bZ$ denotes the multiplicity of the edge $x \to y$, and $\mathcal{W}_\Lambda(x)
\in X^{*}(T)_P$
denotes the $\Lambda$-weight of $x$; see \Cref{def:ref_di} below. 
\end{theorem}
If $w$ is $\pi$-minuscule for a dominant integral weight $\pi$, and if the admissible function
is constant $\Lambda \equiv \pi$ on the interval $[v,w]^P$,
then each multiplicity $m_\Lambda(x_{i},x_{i-1})=1$, and one {obtains an identity
generalizing Nakada's} identity \eqref{E:Nakada-intro}; see \Cref{col:Nak} below. The proof of this ultimately requires a 
good understanding of the (paths in the) interval $[v,w]^P$ for $\pi$-minuscule elements $v,w$, in analogy to 
the intervals in the Young lattice; for instance, the intervals in the weak and strong Bruhat order
coincide. We refer to \S \ref{sec:minuscule-intervals} and \S \ref{sec:genhook} 
for more details. {We encourage the reader to jump to \S \ref{sec:examples}, where we illustrate this theorem in few examples.}

The equality \eqref{E:intro-main} has a remarkable geometric interpretation, observed in special cases in  
\cite{mihalcea:eqqhom,naruse2014schubert,naruse2019skew}.  Let $\csm(Y(v)) \in H^*_T(G/P)$ be the equivariant
Chern--Schwartz--MacPherson (CSM) class of the Schubert variety $Y(v)$. This class is defined using MacPherson's 
construction of characteristic classes of singular varieties \cite{macpherson:chern}, generalized to the equivariant case by 
Ohmoto \cite{ohmoto:eqcsm}; see \S \ref{sec:eqCSM}. Denote by 
\[ \ssm(Y(v)) = \frac{\csm(Y(v))}{c^T(T(G/P))} \in \widehat{H}^*_T(G/P) \]
the (equivariant) Segre--MacPherson (SM) class, which is an element in an 
appropriate completion of the equivariant cohomology ring. Let also 
$\ssm(Y(v))|_w \in H^*_T(pt)_{\loc}$ denote the localization of the SM class at the torus fixed point
$w \in G/P$. {(Here $H^*_T(pt)_{\loc}$ is the fraction field of $H^*_T(pt)$.)} Then we show that: 
\begin{equation}\label{intro:SMfrac} \sum\frac{m_\Lambda(x_r,x_{r-1})}{\mathcal{W}_\Lambda(x_r)}\cdot \frac{m_\Lambda(x_{r-1},x_{r-2})}{\mathcal{W}_\Lambda(x_{r-1})}\cdot \ldots \cdot \frac{m_\Lambda(x_{1},x_{0})}{\mathcal{W}_\Lambda(x_{1})}
= 
\frac{\ssm(Y(v))|_w}{\ssm(Y(w))|_w}\/. \end{equation}
This is proved by utilizing a Molev--Sagan type recursion \cite{molev.sagan:Littlewood-Richardson}, based on a Chevalley formula to multiply SM classes; 
cf.~\Cref{lem:recLR} and see also \cite{su:quantum,AMSS:shadows}. Variants of this recursion have been successfully utilized in \cite{knutson.tao:puzzles,
mihalcea:eqqgr,mihalcea:eqqhom,naruse2014schubert,BCMP:QKChev,naruse2019skew}
to study properties of the {\em equivariant} (quantum) cohomology or K-theory of flag manifolds. As a by-product of this study, and in the same spirit as the references
above, we obtain an algorithm for the structure constants of the multiplication of SM classes; see \Cref{cor:SMLR} below. Different algorithms were also obtained in \cite{su:structure}.

Once~\Cref{intro:SMfrac} is proved, \Cref{thm:intro-main} follows from a smoothness criterion of Schubert varieties 
in terms of localization of SM classes. {More precisely, if $R_w^v$ denotes the 
Richardson variety, the fraction 
\begin{equation*}\label{E:intro-eqmult} \frac{\ssm(Y(v))|_w}{\ssm(Y(w))|_w} = 
e_{w,G/P} (\csm(R_w^{v})) \end{equation*}
equals to Brion's equivariant multiplicity \cite{brion:eq-chow}
of the CSM class of $R_w^v$;~cf.~\Cref{prop:eqmult}. 
{From this perspective, the (generalized) Nakada's formula calculates the equivariant multiplicity of a Richardson variety.}    
The smoothness of $Y(v)$ at $w$ ensures that the equivariant multiplicity is a product of factors corresponding to weights of the normal space of $Y(v)$ at $w$, and it leads to the right hand side of
\Cref{E:intro-main}. Our smoothness criterion for $Y(v)$ at $w$ in \Cref{thm:smoothP} generalizes similar results about smoothness of Schubert varieties by Kumar \cite{kumar1996nil} and Brion \cite{brion:eq-chow}, the latter in terms of equivariant multiplicities. Our criterion is also related to the one from \cite{AMSS:motivic} for motivic Chern classes of Schubert varieties, which had consequences in $p$-adic representation theory.} 

Among the consequences of \Cref{thm:intro-main} proved in section \S \ref{sec:col-and-conseq} 
we mention a skew version of the Peterson formula \eqref{E:introredw}(cf.~\Cref{cor:skewPP}).
\begin{corol} Take $v <w\in W$ be $\pi$-minuscule elements 
such that $Y(v)$ is smooth at $w$. Recall that 
$S(w/v):=\{\beta\in R^+\mid v\le s_\beta w<w\}$.~Then:
\begin{equation}\label{equ:petersonproctor-intro}
\#\Red(wv^{-1})=\frac{(\ell(w)-\ell(v))!}{\prod_{\beta\in S(w/v)}\h(\beta)} ~\/.
\end{equation}
\end{corol}
When considered in this generality, this corollary seems to be new. 
However, we prove that if $W$ is simply laced, the `skew' formula
for $v<w$ equals to the `straight' formula applied to $id < wv^{-1}$. This follows from 
\Cref{smooth:equiv_app}, where we show that $Y(v)$ is smooth at $w$ if and only if
$wv^{-1}$ is $\pi'$-minuscule, for some integral dominant weight $\pi'$. Furthermore,
the $\Lambda$-Bruhat graphs for the intervals $[v,w]$ and $[id,wv^{-1}]$, corresponding to the
parabolic groups stabilizing $\pi$, respectively $\pi'$, are isomorphic.  
The proof utilizes the theory of heaps, defined by Stembridge \cite{stembridge2001minuscule},
together with that of excited diagrams of heaps, studied by Ikeda and Naruse 
\cite{ikeda2009excited} and Naruse and Okada \cite{naruse2019skew}. As a by-product of the proof, we obtain a smoothness criterion in terms of (non)existence of excited diagrams of heaps,
which might be of independent interest.

\Cref{thm:intro-main} is the prototype of more general results. For instance,
a similar theorem - with essentially the same proof - may be obtained if one replaces the (cohomological) SM classes
with their K-theoretic versions, the {\em motivic Segre classes} of Schubert varieties \cite{AMSS:motivic,mihalcea2020left}.
Because of the
technical challenges (in geometry and combinatorics) when working in K-theory, this will be studied
in a separate upcoming work. Furthermore, there are versions of
Nakada's colored hook formula \cite{nakada:colored} which hold in the Kac--Moody setting, suggesting that analogues of  
\Cref{thm:intro-main} might exist in that generality as well.  

{\em Acknowledgments.} Part of this work was performed while L. Mihalcea was in residence at 
ICERM, as a participant in the program ``Combinatorial Algebraic Geometry" (Spring $2021$); the author is grateful for support, excellent
working conditions, and a stimulating atmosphere. L. Mihalcea was also supported in part by a Simons Collaboration Grant. H.~Naruse was supported in part by JSPS KAKENHI Grant Number 16H03921.

\section{Preliminaries}
We start by fixing the notation utilized throughout the paper. Let $G$ be a simply connected complex 
Lie group with Borel subgroup $B$ and maximal torus $T\subset B$. Denote by $\mathfrak g= \mathrm{Lie}(G)$ and by $\fh= \mathrm{Lie}(T)$ 
be corresponding Lie algebras. Let $R^+\subset \fh^*:=\fh^*_\bbQ$ denote the positive roots, 
i.e those roots in $B$, and by $\Sigma =\{ \alpha_i: i \in I \}$ the set of simple roots. Let $R:=R^+\sqcup -R^+$. 
We use $\alpha>0$ (resp. $\alpha<0$) to denote $\alpha\in R^+$ (resp. $\alpha\in -R^+$). For any 
root $\alpha\in R$, let $\alpha^\vee \subset \fh$ denote the corresponding coroot. Let 
$\langle \cdot,\cdot \rangle: \frak{h}^{*}\times \frak {h}\to \bbQ$ denote the usual pairing, and let $X^*(T)\subset \fh^*$ be the weight lattice. 
The Weyl group $W = N_G(T)/T$ is generated by simple reflections $s_i= s_{\alpha_i}$ ($i \in I$), and it is equipped with 
the Bruhat order $\leq$; we denote by $w_0$ the longest element. For any $w\in W$, let $\Red(w)$ 
denote the set of all the reduced expressions for $w$. For $v < w \in W$, define 
\begin{equation}\label{equ:wvinversion}
S(w/v):=\{\beta\in R^+\mid v\le s_\beta w<w\}.
\end{equation}
If $v=id$, we denote $S(w/id)$ by $S(w)$.

Let $P (\supseteq B)$ be a parabolic subgroup with simple roots $\Sigma_P\subset \Sigma$ in $P$. This determines
the set $R_P^+ \subset R^+$ of those positive roots spanned by $\Sigma_P$, and the subgroup $W_P \subset W$ generated 
by the simple reflections $s_i$ where $\alpha_i \in \Sigma_P$. Denote by $W^P$ the set of 
minimal length representatives in $W/W_P$. The elements $w\in W^P$ are characterized by the property that $w(R^+_P)\subset R^+$. 
The torus fixed points $(G/P)^T$ are $\{wP\mid w\in W^P\}$. For any $w\in W^P$, let 
$T_w(G/P)$ denote the tangent space at $wP$. For any $w\in W$, let $X(w)^\circ:=BwP/P\subset G/P$ (resp. 
$Y(w)^\circ:=B^-wP/P\subset G/P$) denote the Schubert cell with closure $X(w)$ (resp. $Y(w)$), where 
$B^-$ is the opposite Borel subgroup. In particular, $X(w)^\circ=X(u)^\circ$ when the two cosets $wW_P$ and 
$uW_P$ are equal to each other. The Bruhat order restricts to the (Bruhat) order
on the cosets 
$\{wW_P\mid w\in W^P\}$, and it characterized by $uW_P\leq wW_P$ if and only if 
$uP\in X(w)$. Let $X^{*}(T)_P:=\{\lambda\in X^{*}(T)\mid \langle\lambda, \gamma^\vee\rangle=0 \text{ for all }
\gamma\in R^+_P \}$ be the set of integral weights which vanish on $ (R^+_P)^\vee$. For any 
$\lambda\in X^*(T)_P$, let $\calL_\lambda$ denote the line bundle $G\times^P\bbC_\lambda\in\Pic(G/P)$, which has 
fibre over $1.P$ the $T$-module of weight $\lambda$.

Let $H_T^*(G/P)$ denote the equivariant cohomology of the partial flag variety $G/P$. It has a basis of the Schubert classes:
\[H_T^*(G/P)=\bigoplus_{w\in W^P} H_T^*(pt)[X(w)]=\bigoplus_{w\in W^P} H_T^*(pt)[Y(w)],\]
{where $[X(w)]$ and $[Y(w)]$ denote the Poincar{\'e} dual of the fundamental classes of the Schubert varieties.}
For any $\kappa \in H_T^*(G/P)$ and $w\in W^P$, let $\kappa|_w\in H_T^*(pt)$ denote the restriction of $\kappa$ to the fixed point $wP\in G/P$.
Let $H_T^*(G/P)_{\loc} :=H_T^*(G/P)\otimes_{H_T^*(pt)}\Frac H_T^*(pt)$ be the localized equivariant cohomology of $G/P$, where $\Frac H_T^*(pt)$ denotes the fraction field of $H_T^*(pt)$. 
 By the localization theorem,  $H_T^*(G/P)_{\loc}$ has a basis formed by the classes of fixed points $\{[wP]\mid w\in W^P\}$; 
this is called the fixed point basis.  
 
The Weyl group $W$ acts on $G/P$ by left multiplication. It induces an action of $W$ on $H_T^*(G/P)$, 
which acts on the base ring 
$H_T^*(pt)=\bbZ[\fh]$ by the usual Weyl group action. Let $\varphi_{w_0}$ denote the action induced by 
the longest Weyl group element. Then $\varphi_{w_0}([X(w)])=[Y(w_0w)]$.

\section{Nakada's colored hook formula}\label{sec:hook}
In this section, we review Nakada's colored hook formula from \cite{nakada:colored,nakada:proc}. This formula generalizes results of Proctor 
\cite{proctor:minuscule,proctor:dynkin}
and D. Peterson (see e.g. \cite{carrell:vector})
about the combinatorics of complete $d$-posets, and the number of reduced 
decompositions of certain minuscule Weyl group elements; see also \cite{stanley:reduced}. 
In order to be able to utilize geometric arguments, in this paper we require that the Lie algebra $\fg$ is of finite 
type, although Nakada's formula may be formulated for any Kac--Moody Lie algebra.

\subsection{Nakada's formula and pre-dominant weights}\label{sec:predominant} Recall that the integral weights $\lambda \in \frak{h}^*$ satisfy 
$\langle\lambda,\alpha_i^\vee\rangle\in \bZ$, for each $\alpha_i \in \Sigma$; a weight $\lambda$ is {\em dominant} if in addition $\langle\lambda,\alpha_i^\vee\rangle\geq 0$. Following 
\cite{nakada:colored}, we say that an integral weight $\lambda$ is {\em pre-dominant} if
$\langle \lambda, \beta^\vee\rangle \geq -1$ for all positive roots $\beta \in R^{+}$. For a pre-dominant integral weight $\lambda$, the {\em diagram} $D(\lambda)$ of $\lambda$ is defined by
\begin{equation}\label{diagram}
D(\lambda):=\{\beta\in R^{+} \mid \langle\lambda, \beta^\vee\rangle =-1\}.
\end{equation}
\noindent
If $\lambda$ is a dominant integral weight, it is pre-dominant, but $D(\lambda)=\emptyset$. 

We recall some elementary facts about pre-dominant integral weights from \cite{nakada:colored}. 
\begin{lemma}\cite[Lemma 4.1]{nakada:colored}\label{lem:pre-dom}
Let $\lambda$ be a pre-dominant integral weight.

(1) If $D(\lambda)\neq\emptyset$, then $D(\lambda) \cap \Sigma\neq \emptyset$.

(2) For $\beta\in D(\lambda)$, $s_\beta(\lambda)$ is a pre-dominant integral weight.

(3) In case (2), $D(s_\beta(\lambda))=s_\beta(D(\lambda)\setminus S(s_\beta))$.
\end{lemma}

For integral weights $\mu, \nu \in \frak{h}^*$ and $\beta \in R^{+}$, define
\[\mu \overset{\beta}{\to} \nu \Longleftrightarrow \langle \mu, \beta^\vee\rangle=-1 \textit{ and } \nu=s_\beta(\mu) \/.\] In particular, if
$\mu \overset{\beta}{\to} \nu$ then $\nu = \mu + \beta$.

A $\lambda$-path of length $r$ is a sequence $(\beta_1,\beta_2,\ldots, \beta_r)$ where 
$$
\lambda=\lambda_0 \overset{\beta_1}{\to} \lambda_1 
\overset{\beta_2}{\to} \cdots
\overset{\beta_r}{\to} \lambda_r \/.
$$
We denote by ${\rm Path}(\lambda)$ the set of all $\lambda$-paths.
By \Cref{lem:pre-dom}(2), if $\lambda$ is pre-dominant, all the weights $\lambda_i$ in a $\lambda$-path are
pre-dominant. At each step in a $\lambda$-path, $|D(\lambda_i)|$ strictly decreases, therefore the 
length of a $\lambda$-path for a pre-dominant weight must be at most the size of $D(\lambda)$. 
Parts (1) and (3) of \Cref{lem:pre-dom} imply that a $\lambda$-path $(\beta_1, \ldots, \beta_d)$ of maximal length can only contain simple roots $\beta_i \in \Delta$, and that $d = |D(\lambda)|$. For a pre-dominant integral weight $\lambda$, 
we denote by ${\rm MPath}(\lambda)\subset {\rm Path}(\lambda)$ the subset of longest $\lambda$-paths.

Now we can state Nakada's colored hook formula.
\begin{theorem}[Nakada's colored hook formula] \cite[Theorem 7.1]{nakada:colored}\label{thm:Nak}

Let $\lambda$ be a pre-dominant integral weight. Then 
$$
\sum \frac{1}{\beta_1}\cdot \frac{1}{\beta_1+\beta_2}\cdot \ldots \cdot\frac{1}{\beta_1+\beta_2+\cdots +\beta_r}
=\prod_{\beta\in D(\lambda)}\left(1+\frac{1}{\beta}\right) \/, 
$$ where the sum is over all $r \ge 0$ and $\lambda$-paths $(\beta_1,\beta_2,\ldots,\beta_r)$.
\end{theorem}

Taking the lowest degree terms of the formula in Theorem \ref{thm:Nak},
we get the following corollary.
\begin{corol}\cite[Corollary 7.2]{nakada:colored}
Let $\lambda$ be a pre-dominant integral weight and $d=|D(\lambda)|$.
Then 
$$
\sum_{
(\beta_{1},\beta_{2},\ldots,\beta_d)\in {\rm MPath}(\lambda)}
\frac{1}{\beta_1}\cdot \frac{1}{\beta_1+\beta_2}\cdot \ldots \cdot\frac{1}{\beta_1+\beta_2+\cdots +\beta_d}
=\prod_{\beta\in D(\lambda)}\frac{1}{\beta}.
$$
\end{corol}
This formula generalizes the classical hook formula \eqref{equ:hook} and the Peterson--Proctor formula \eqref{equ:petersonproctor-intro}; see also \S \ref{sec:genhook} below. It is also related to an equality of rational functions involving root partitions for cluster variables cf. \cite{Cas19,casbi2019newton}.

\subsection{Peterson minuscule elements}\label{sec:minuscule} 
For any integral weight $\pi$, D. Peterson defined the notion of a $\pi$-minuscule 
Weyl group element, see below and  \cite{proctor:bruhat,carrell:vector,proctor:minuscule,stembridge:fullycomm,stembridge2001minuscule,graham.kreiman:cominuscule_pt}.
Examples of $\pi$-minuscule elements are the minimal length representatives in $W^P$, where
$P$ is the maximal parabolic associated to a minuscule fundamental weight.~{Minuscule elements are fully commutative
\cite{stembridge:fullycomm}; in particular, in type A they are $321$-avoiding.}
We will utilize this notion to rewrite Nakada's formula and its generalizations considered in this paper in terms of Weyl group elements. 
\begin{defin}[$\pi$-minuscule elements]\label{def:minuscule}
Let $\pi$ be an integral weight.
An element $w\in W$ is called {\bf $\pi$-minuscule} if
there is a reduced expression $w=s_{i_1} s_{i_2}\cdots s_{i_\ell}$ such that
\begin{equation}\label{minuscule:1}
\langle s_{i_{k+1}} s_{i_{k+2}}\cdots s_{i_\ell}(\pi), \alpha_{i_k}^\vee \rangle=1\hspace{1cm} (1\leq k\leq \ell)
\end{equation}
Equivalently, 
\begin{equation}\label{equ:actiononmin}
s_{i_k} s_{i_{k+1}} s_{i_{k+2}}\cdots s_{i_\ell}(\pi) = \pi - \alpha_{i_\ell} - \ldots - \alpha_{i_k}.
\end{equation}
\end{defin}
From definition it follows that if $w$ is $\pi$-minuscule, then its length is $\ell(w) = \h(\pi- w(\pi))$.  
The next result follows immediately from analyzing the inversion set of $w$. 
\begin{lemma}\label{lem:gamma} \cite[Proposition 5.1]{stembridge2001minuscule}
If $w\in W$ then
\begin{center}
$w$ is $\pi$-minuscule $\iff$ $\langle\pi, \gamma^\vee\rangle=1$ for all $\gamma\in R^{+}$
such that $w s_\gamma <w$.
\end{center}
\end{lemma}
From now on we restrict to the case when $\pi$ is a a dominant integral weight. Let $W_\pi=\Stab_W(\pi)$ denote the stabilizer subgroup of $\pi$ inside $W$. This determines the parabolic subgroup $P$ such that $W_P=W_\pi$, containing simple roots $\Sigma_P:=\{\alpha_i\in \Sigma\mid \langle \pi,\alpha_i^\vee\rangle=0\}$. Let $W^\pi$ denote the set of minimal length representatives for the cosets $W/W_\pi$.
\begin{remark}\label{rem:wminuscule}
It follows from \Cref{lem:gamma} that if $w$ is $\pi$-minuscule, then $w\in W^\pi$,
and
the property (\ref{minuscule:1}) holds for any reduced expression 
of $w=s_{i_1} s_{i_2}\cdots s_{i_\ell}$.
\end{remark}
We also need the following definition.
\begin{defin}\label{defin:order}
For $u,w\in W^\pi$ and $\beta\in R^{+}$, define 
\begin{equation*}\label{equ:order}
u \overset{\beta}{\to} w \textit{\quad if \quad}  s_\beta wW_\pi=uW_\pi, \textit{ and }u<w.
\end{equation*}
\end{defin}
\begin{lemma}\label{lemma:ubetav} Let $\pi$ be dominant integral and $u,w \in W^\pi$ such that 
$u \overset{\beta}{\to} w$.~Then the following hold:

(a) The root $\beta$ is unique.

(b) There exists a unique positive root $\gamma$ such that $s_\beta w W_P = w s_\gamma W_P$.
Furthermore, 
\[ \gamma= - w^{-1}(\beta) \quad \textrm{ and } \quad ws_\gamma < w \/. \]
%The root $\gamma= - w^{-1}(\beta)$ and it satisfies $ws_\gamma < w$.

(c) If in addition $w$ is $\pi$-minuscule then $\langle w(\pi), \beta^\vee \rangle = -1$.
\end{lemma}
\begin{proof} The uniqueness of $\beta$ follows from \cite[Lemma 4.1]{fulton.woodward}. Since 
$s_\beta w <w$ it follows that $w^{-1}(\beta)<0$, thus $\gamma:= - w^{-1}(\beta)$ is a positive root,
and $s_\gamma W_P = w^{-1} s_\beta w W_P = w^{-1} u W_P$. Since $w^{-1}u \notin W_P$ by hypothesis, the uniqueness of $\gamma$ follows from \cite[Lemma 2.2]{buch.m:nbhds}. 
Finally, since $ws_\gamma W_P = u W_P$ and $u <w$, then necessarily
$ws_\gamma < w$. This finishes the proof of (b). Part (c) follows from (b) and \Cref{lem:gamma}.
\end{proof}

The relation between pre-dominant integral weights and $\pi$-minuscule elements is given by the following proposition.
\begin{prop}\label{prop:pre-dom}\cite[Propositions 10.1 and 10.3]{nakada:colored}
There is a bijection between the following two sets
\begin{align*}
\{\textit{pre-dominant integral weights } \lambda\}&\\
\longleftrightarrow  \{(\pi, w)\mid & \pi \textit{ is dominant integral, and }w \textit{ is } \pi \textit{-minuscule}\}.
\end{align*}
Here, $\lambda$ is determined by $(\pi,w)$ by the formula $\lambda=w(\pi)$.
Conversely, for any pre-dominant integral weight $\lambda$,
take a maximal $\lambda$-path 
$(\alpha_{i_1},\alpha_{i_2},\ldots,\alpha_{i_d})\in {\rm MPath}(\lambda)$ and
set $w=s_{i_1} s_{i_2}\cdots s_{i_d}$, $\pi=w^{-1}(\lambda)$.
Then $\pi$ is a dominant integral weight and $w$ is $\pi$-minuscule.
Moreover $D(\lambda)= S(w)=\{\beta\in R^{+}\mid s_\beta w<w\}$, and the correspondence $(\alpha_{i_1},\alpha_{i_2},\ldots,\alpha_{i_d})\mapsto s_{i_1} s_{i_2}\cdots s_{i_d}$ from ${\rm MPath}(\lambda)$ to $\Red(w)$ is bijective.
\end{prop}

\begin{corol}\label{lem:wminuscule}
Let $\pi$ be a dominant integral weight, $w$ be $\pi$-minuscule and $\lambda=w(\pi)$. For $\beta\in R^+$, let $u\in W^\pi$ denote the minimal length representative in $s_\beta wW_\pi$. Then: 

(a) $\lambda\overset{\beta}{\to} s_\beta(\lambda)$ is equivalent to $u \overset{\beta}{\to} w$.

(b) If $u \overset{\beta}{\to} w$, then $u$ is also $\pi$-minuscule.
\end{corol}
\begin{proof} Part (a) follows from:
\[\lambda\overset{\beta}{\to} s_\beta(\lambda)\Leftrightarrow \langle\lambda,\beta^\vee\rangle=-1\Leftrightarrow \langle\pi,w^{-1}(\beta^\vee)\rangle=-1 \Leftrightarrow w^{-1}\beta<0\Leftrightarrow  s_\beta w<w\Leftrightarrow u\overset{\beta}{\to} w.\]
Here in the third equivalence, the $\Rightarrow$ direction follows from the fact that $\pi$ is dominant, while the $\Leftarrow$ direction follows from \Cref{lemma:ubetav}. 

To prove (b), observe that if $u\overset{\beta}{\to} w$, then $s_\beta\lambda$ is also a pre-dominant integral weight by \Cref{lem:pre-dom}(2). Because of \Cref{prop:pre-dom},
\[s_\beta\lambda=w'(\pi')\]
for some dominant integral weight $\pi'$, $w'\in W$ such that $w'$ is $\pi'$-minuscule. Therefore,
\[s_\beta w(\pi)=w'(\pi') \/,\]
 and because the dominant 
weights $\pi, \pi'$ are in the fundamental 
domain for the $W$-action, it follows that $\pi=\pi'$, and that $s_\beta wW_\pi=w'W_\pi$. By \Cref{rem:wminuscule}, $w'\in W^\pi$. Thus, $w'=u$, and $u$ is $\pi$-minuscule.
\end{proof}

With \Cref{lem:wminuscule}, Nakada's formula (\Cref{thm:Nak}) can be reformulated as follows.
\begin{theorem}[Nakada]\label{thm:anotherform}
Let $\pi$ be a dominant integral weight, and $w\in W$ be a $\pi$-minuscule element. Then: 
\[ \sum \frac{1}{\beta_1}\cdot \frac{1}{\beta_1+\beta_2}\cdot \ldots \cdot\frac{1}{\beta_1+\beta_2+\cdots +\beta_r}
=\prod_{\beta\in S(w)}\left(1+\frac{1}{\beta}\right) \/,
\] where the summation is over $r \ge 0$ and paths $x_r\overset{\beta_r}{\to} x_{r-1}\to \ldots \to x_1\overset{\beta_1}{\to} x_0=w$.
In particular, 
$$
{\sum_{x_d\overset{\beta_d}{\to} x_{d-1}\to \ldots \to x_1\overset{\beta_1}{\to} x_0=w\/}}
\frac{1}{\beta_1}\cdot \frac{1}{\beta_1+\beta_2}\cdot \ldots \cdot \frac{1}{\beta_1+\beta_2+\cdots +\beta_d}
=\prod_{\beta\in S(w)}\frac{1}{\beta},
$$
where $d=|S(w)|{= \ell(w)}$.
\end{theorem}
This statement will generalized in \Cref{col:Nak} below.

\subsection{Bruhat intervals of $\pi$-minuscule elements}\label{sec:minuscule-intervals}
The main goal of this section is to prove \Cref{prop:weakstrong}, stating that the intervals in the weak and ordinary (or strong) Bruhat
order determined by two $\pi$-minuscule elements coincide. Special cases appeared
in the works by Stembridge \cite[Thm.~7.1]{stembridge:fullycomm} (for $\pi$ a minuscule weight) and Proctor \cite[\S 10]{proctor:minuscule} (for $G$ simply laced),
but we have not seen \Cref{prop:weakstrong} in the generality we need.
  
For $u,w \in W$ the (left) weak Bruhat order is defined by $u<_L w$ iff 
$\ell(wu^{-1}) = \ell(w) - \ell(u)$. Equivalently, $w=vu$ and $\ell(vu) = \ell(v) + \ell(u)$. 
Observe that if $u<_L w$ then $u<w$ in the strong Bruhat order, but in general the 
two orders are different. For $v<w\in W^P$, set $[v, w]^P:=\{x\in W^P\mid v\leq x\leq w\}$. 
\begin{prop}\label{prop:weakstrong} Let $\pi$ is a dominant integral weight 
with $\Stab_W(\pi)=W_P$
and let $u<w$ be $\pi$-minuscule elements in $W^P$. Then the interval
$[u,w]^P \subset W^P$ is the same in the weak and strong Bruhat orders, i.e.,
\[ \{ v \in W^P: u \le v \le w \} = \{ v' \in W^P: u \le_L v' \le_L w \} \/. \] 
\end{prop}
\begin{proof} Let $v \in W^P$ such that $v < w$. Then it suffices to show that 
$v \le_L w$ when $w$ covers $v$ in the strong Bruhat order in $W^P$, i.e., 
$\ell(w) - \ell(v) =1$. Write
$v W_P= s_\beta w W_P = w s_\gamma W_P$ for positive roots $\beta, \gamma$ as in 
\Cref{lemma:ubetav}. Since $w$ is $\pi$-minuscule, $v$ is also $\pi$-minuscule 
by \Cref{lem:wminuscule} (b).
Using that $\ell(w)= \h(\pi - w(\pi))$ we obtain 
\[ 1 = \ell(w) - \ell(v) = \mathrm{ht}(s_\beta w(\pi) - w(\pi)) = 
\h(w(\pi) - \langle w(\pi), \beta^\vee\rangle \beta - w(\pi)) = \h(\beta) \/.\]
The last equality follows because the multiplicity $\langle w(\pi), \beta^\vee\rangle =-1$ 
by \Cref{lemma:ubetav}(c). Therefore $\beta$ must be a simple root, and, furthermore, $w=s_\beta v$,
thus $v \le_L w$. 
\end{proof}
The special case $u=id$ of the following corollary has been proved by Nakada
\cite[Proposition 10.3]{nakada:colored}.
\begin{corol}\label{cor:skewred} Assume the hypotheses from \Cref{prop:weakstrong}
and let $d:= \ell(w)-\ell(u)$.
There is a bijection between the set of reduced words of $wu^{-1}$ and the set of maximal length
paths from $u$ to $w$, {sending a reduced word $wu^{-1} = s_{\beta_1} \cdot \ldots \cdot s_{\beta_d}$ to} the path
\[ u=x_d\overset{\beta_d}{\to} x_{d-1}\to \ldots \to x_1\overset{\beta_1}{\to} x_0=w \/. \]
\end{corol}
\begin{proof} Consider any path 
$u=x_d\overset{\beta_d}{\to} x_{d-1}\to \ldots \to x_1\overset{\beta_1}{\to} x_0=w$
in the (strong) Bruhat interval $[u,w]^P$. Then $\ell(x_{i-1}) - \ell(x_{i}) = 1$ for all $i$, and
since $w$ is $\pi$-minuscule all $x_i$'s must be also. Because weak and strong 
Bruhat orders coincide in $[u,w]^P$ by \Cref{prop:weakstrong},
$x_{i} = s_{\beta_i} x_{i-1}$ and each of the roots $\beta_i$ must be simple. Then
$s_{\beta_d}s_{\beta_{d-1}} \cdot \ldots \cdot s_{\beta_1}= uw^{-1}$ and this decomposition
is reduced. This associates a reduced word of $wu^{-1}$ to the given path (by reading in reverse), 
and it easily follows that this correspondence is a bijection.\end{proof}

\section{Chern--Schwartz--MacPherson classes of Schubert cells}\label{sec:eqCSM}
In this section, we recall the definition of the equivariant Chern--Schwartz--MacPherson (CSM) and Segre--MacPherson (SM) classes, 
then we recall the Chevalley formula for the CSM and SM classes of Schubert cells in partial flag manifolds \cite{AMSS:shadows,su:quantum}.

\subsection{Definition}
Let $X$ be a complex algebraic variety. The group of constructible functions $\calF(X)$ consists of functions $\varphi = \sum_W c_W \one_W$, where the sum is over a finite set of constructible subsets $W \subset X$,
 $c_W \in \bbZ$ are integers, and $\one_W$ is the characteristic function of $W$.
 For a proper morphism $f:Y\rightarrow X$, there is a linear map $f_*:\calF(Y)\rightarrow \calF(X)$, such that for any constructible subset $W\subset Y$, $f_*(\one_W)(x)=\chi_{\textit{top}}(f^{-1}(x)\cap W)$, where $x\in X$ and $\chi_{\textit{top}}$ denotes the topological Euler characteristic.
Thus $\calF$ can be considered as a (covariant) functor from the category of complex algebraic varieties and 
proper morphisms to the category of abelian groups.

According to a conjecture attributed to Deligne and Grothendieck, there is a unique natural transformation $c_*: \calF \to H_*$ from the functor of constructible functions on a complex algebraic variety $X$ to the homology functor, where all morphisms are proper, such that if $X$ is smooth then $c_*(\one_X)=c(TX)\cap [X]$, {where $c(TX)$ denotes the total Chern class of the tangent bundle $TX$ and $[X]$ denotes the fundamental class}.  This conjecture was proved by MacPherson \cite{macpherson:chern}; the class $c_*(\one_X)$ for possibly singular $X$ was shown to coincide with a class defined earlier by M.-H.~Schwartz \cite{schwartz:1, schwartz:2, BS81}. 

The theory of CSM classes was later extended to the equivariant setting by Ohmoto \cite{ohmoto:eqcsm}. If $X$ has an 
action of a torus $T$, Ohmoto defined the group $\calF^T(X)$ of {\em equivariant\/} constructible functions. 
We will need the following properties of this group: 
\begin{enumerate} 
\item If $W \subseteq X$ is a constructible set which is invariant under the $T$-action, its characteristic function $\one_W$ is an element of $\calF^T(X)$. We will denote by $\calF_{inv}^{T}(X)$ the subgroup of $\calF^{T}(X)$ consisting of $T$-invariant constructible functions on $X$. (The group $\cF^T(X)$ also contains other elements, but this will be immaterial for us.)
\item Every proper $T$-equivariant morphism $f: Y \to X$ of algebraic varieties induces a homomorphism $f_*^T: \calF^T(X) \to \calF^T(Y)$. The restriction of $f_*^T$ to $\calF_{inv}^{T}(X)$ coincides with the ordinary push-forward $f_*$ of constructible functions. See \cite[\S 2.6]{ohmoto:eqcsm}.
\end{enumerate}

Ohmoto proves \cite[Theorem 1.1]{ohmoto:eqcsm} that there is an equivariant version of MacPherson transformation $c_*^T: \calF^T(X) \to H_*^T(X)$ that satisfies $c_*^T(\one_X) = c^T(TX) \cap [X]_T$ if $X$ is a non-singular variety, and that is functorial with respect to proper push-forwards. The last statement means that for all proper $T$-equivariant morphisms $Y\to X$ the following diagram
commutes:
$$\xymatrix{ 
\calF^T(Y) \ar[r]^{c_*^T} \ar[d]_{f_*^T} & H_*^T(Y) \ar[d]^{f_*^T} \\ 
\calF^T(X) \ar[r]^{c_*^T} & H_*^T(X).}$$ 
If $X$ is smooth, we will identify the (equivariant) homology and cohomology groups, 
by Poincar{\'e} duality: $H_*^T(X)\simeq H^*_T(X)$.

\begin{defin}
Let $Z$ be a $T$-invariant constructible subvariety of $X$. 
\begin{enumerate}
\item
We denote by $\csm(Z):=c_*^T(\one_{Z}) {~\in H_*^T(X)}$ the {\em equivariant Chern--Schwartz--MacPherson (CSM) class\/} of $Z$. 
\item 
If $X$ is smooth, we denote by $\ssm(Z):=\frac{c_*^T(\one_{Z})}{c^T(T X)} ~\in \widehat{H}^*_T(X)$ 
%{~\in H^*_T(X)_{\loc}}$ 
the {\em equivariant Segre--MacPherson (SM) class\/} of $Z$, where $\widehat{H}^*_T(X)$ is an appropriate completion
of $H^*_T(X)$. 
\end{enumerate}
\end{defin}

\subsection{Chevalley formulae for CSM and SM classes of Schubert cells}
In this section, we recall the Chevalley formula for the CSM/SM classes of the Schubert cells in the partial flag variety $G/P$,
proved in \cite{AMSS:shadows,su:quantum}.

Let $H^*_T(X)_{\loc}:=H^*_T(X)\otimes_{H^*_T(pt)} \Frac H^*_T(pt)$ denote the localization of $H^*_T(X)$, where 
$\Frac H^*_T(pt)$ is the fraction field of $H^*_T(pt)$. Then by the localization formula, CSM classes $\{\csm(X(w)^\circ)\mid w\in W^P\}$ and SM classes $\{\ssm(Y(w)^\circ)\mid w\in W^P\}$ are bases for $H_T^*(G/P)_{\loc}$, dual under the Poincar\'e pairing $\langle-,-\rangle_{G/P}$ on $H_T^*(G/P)_{\loc}$. This means that:
\begin{equation*}
\langle \ssm(Y(u)^\circ), \csm(X(w)^\circ)\rangle_{G/P}=\delta_{u,w} \textit{ for any } w,u\in W^P.
\end{equation*}
Moreover, if $\varphi_{w_0}$ denotes the automorphism of $H_T^*(G/P)$ induced by the left multiplication by $w_0$ on $G/P$,
then
\begin{equation}\label{equ:w0action}
\varphi_{w_0}(\csm(X(w)^\circ))=\csm(Y(w_0w)^\circ) \/.
\end{equation}
\begin{theorem}\cite{AMSS:shadows,su:quantum}\label{thm:chevalley} 
For any $w\in W^P$ and $\lambda\in X^*(T)_P$, the following holds in $H_T^*(G/P)$:
\begin{equation*}\label{equ:ChePX}
c_1(\calL_\lambda)\cup \csm (X(w)^\circ)=w(\lambda) \csm (X(w)^\circ)-\sum_{\alpha>0,ws_\alpha <w}\langle\lambda,\alpha^\vee\rangle \csm (X(ws_\alpha)^\circ),
\end{equation*}
and
\begin{equation*}\label{equ:cheL}
c_1(\calL_\lambda)\cup \ssm (Y(w)^\circ)=
w(\lambda) \ssm (Y(w)^\circ)-\sum_{\alpha>0,ws_\alpha >w}\langle\lambda,\alpha^\vee\rangle {\ssm (Y(ws_\alpha)^\circ)} \/.
\end{equation*}
\end{theorem}
\begin{proof}
The first equality follows from \cite[Theorem 3.7]{su:quantum} and \cite{AMSS:shadows}. We prove the second one. Applying $\varphi_{w_0}$ to the first equation, we get
\[c_1(\calL_\lambda)\cup \csm (X(w_0w)^\circ)=w_0w(\lambda) \csm (X(w_0w)^\circ)-\sum_{\alpha>0,ws_\alpha <w}\langle\lambda,\alpha^\vee\rangle \csm (X(w_0ws_\alpha)^\circ).\]
Here we have used \Cref{equ:w0action} and the fact that $c_1(\calL_\lambda)$ is $W$-invariant. 
Assume $w_0w=zu$ for some $z\in W^P$ and $u\in W_P$. Then $w_0ws_\alpha W_P=zus_\alpha W_P=zs_{u\alpha}W_P$, and we have the following equivalences
\begin{align*}
&\alpha>0, ws_\alpha <w\\
\Leftrightarrow&\alpha>0, w\alpha<0\\
\Leftrightarrow&\alpha\in R^+\setminus R^+_P, w\alpha<0\\
\Leftrightarrow&\beta:=u\alpha\in R^+\setminus R^+_P, z\beta>0\\
\Leftrightarrow&\beta\in R^+\setminus R^+_P, zs_\beta>z.
\end{align*}
Moreover, $\langle\lambda,\alpha^\vee\rangle=\langle\lambda,\beta^\vee\rangle$ since $\lambda\in X^*(T)_P$.
Finally, the second equation holds by the above equivalences and the fact that $\langle\lambda,\gamma^\vee\rangle=0$ for any $\gamma\in R^+_P$.
\end{proof}

\section{Segre--MacPherson Littlewood--Richardson (SMLR) coefficients}
For any $u,v,w\in W^P$, define the Segre--MacPherson Littlewood--Richardson (SMLR) coefficients 
$d_{u,w}^{v} {\in H^*_T(pt)_{\loc}}$ for the SM classes of Schubert cells by
\begin{equation}\label{equ:CSMLR}
\ssm({Y(u)^\circ}) \cup \ssm({Y(v)^\circ}) =\sum_{w\in W^P} d_{u,v}^{w}\ssm({Y(w)^\circ})\in {
H_T^*(G/P)_{\loc}} \/.
\end{equation}

A formula for the structure constants $d_{u,v}^{w}$, in terms of multiplications in the cohomology
of Bott-Samelson varieties, 
was recently obtained by the third-named author in \cite[Theorem 5.2]{su:structure}. 
In what follows we will obtain a recursive procedure to calculate the coefficients, based on the
equivariant Chevalley formula for the SM classes. Instances of this
recursion in various equivariant (quantum) cohomology 
and K theory rings,
appeared in \cite{molev.sagan:Littlewood-Richardson,knutson.tao:puzzles,
mihalcea:eqqgr,mihalcea:eqqhom,naruse2014schubert,BCMP:QKChev,naruse2019skew}. 

We start with the following simple lemma.
\begin{lemma}\label{lem:CSMLR} The following properties hold
for the SMLR coefficients $d_{u,v}^w$: 
\begin{center}
\begin{itemize}
\item[(a)] If $d_{u,v}^{w}\neq 0$, then $u\leq w$
and $v\leq w$.
\item[(b)] $d_{u,v}^{v}= \ssm({Y(u)^\circ})|_v$.
\end{itemize}
\end{center}
\end{lemma}
\begin{proof} To prove (a), localize both sides at the fixed point $wP$, and observe
that \[ \ssm({Y(u)^\circ})|_{w}=0 \] unless $u\leq w$. Part (b) follows again by localization, 
after restricting both sides of \Cref{equ:CSMLR} to the fixed point $vP$.
\end{proof}

We record the following explicit localization formula for $\ssm({Y(u)^\circ})|_v$, generalizing the one in the complete flag variety case \cite[Corollary 6.7]{AMSS:shadows}.
\begin{prop}\label{prop:loc}
Let $u\leq v\in W^P$.
Fix a reduced expression $v=s_{i_1} s_{i_2} \cdots s_{i_\ell}$
and set $\beta_j=s_{i_1} s_{i_2} \cdots s_{i_{j-1}}(\alpha_{i_j})$
$(j=1,2,\ldots,\ell)$.
Then
$$\ssm (Y(u)^\circ)|_v=\frac{1}{\prod_{1\leq j\leq \ell}(1+\beta_j)}
\sum \beta_{j_1}\beta_{j_2}\cdots \beta_{j_k},$$
where the summation is over $ 1\leq j_1<j_2<\cdots j_k\leq \ell$ such that 
$s_{i_{j_1}}s_{i_{j_2}}\cdots s_{i_{j_k}}W_P=uW_P$.
\end{prop} 
\begin{proof}
Let $p:G/B\rightarrow G/P$ be the natural projection. In this proof, we use $Y(w)^\circ_B:=B^-wB/B\subset G/B$ to denote the Schubert cells in $G/B$. 
By \cite[Equation (37)]{AMSS:shadows}, we have
\[p^*(\ssm (Y(u)^\circ))=\sum_{z\in W_P}\ssm (Y(uz)_B^\circ).\]
Restricting both sides to the fixed point $vB\in G/B$, we get
\begin{align*}
\ssm (Y(u)^\circ)|_{v}&=p^*(\ssm (Y(u)^\circ))|_{vB}\\
&=\sum_{z\in W_P}\ssm (Y(uz)_B^\circ)|_{vB}\\
&=\frac{\prod_{\alpha>0,v\alpha>0}(1-v\alpha)}{\prod_{\alpha>0}(1-v\alpha)}\sum \beta_{j_1}\beta_{j_2}\cdots \beta_{j_k}\\
&=\frac{1}{\prod_{1\leq j\leq \ell}(1+\beta_j)}
\sum \beta_{j_1}\beta_{j_2}\cdots \beta_{j_k}.
\end{align*}
Here the third equality follows from \cite[Corollary 6.7]{AMSS:shadows}, and the summations in the last two lines are over $ 1\leq j_1<j_2<\cdots j_k\leq \ell$ such that 
$s_{i_{j_1}}s_{i_{j_2}}\cdots s_{i_{j_k}}W_P=uW_P$.
\end{proof}

For any $\lambda\in X^*(T)_P$ and $u,w\in W^P$, define the Chevalley coefficients $c^u_{\lambda,w}$ by the equation
\begin{equation}\label{equ:ccoeff}
c_1(\calL_\lambda)\cup \ssm (Y(w)^\circ)=
\sum_{u\in W^P} c_{\lambda,w}^{u} \ssm (Y(u)^\circ).
\end{equation}
These coefficients are calculated in \Cref{thm:chevalley}. The following is 
our main recursion formula for $d_{u,v}^w$.
\begin{prop}
\label{lem:recLR}
For any $u,v,w\in W^P$ and any weight $\lambda\in X^*(T)_P$, 
the following holds:
$$
(c^w_{\lambda,w} - c^u_{\lambda,u}) d^w_{u,v}=
\sum_{\substack{u< x}} 
c^x_{\lambda,u}d^w_{x,v}-\sum_{\substack{ y<w}} c^w_{\lambda,y} d^y_{u,v} \/.
$$
\end{prop}
\begin{proof}
By the associativity of the cup product, we have
\[
\bigg( c_1(\calL_\lambda) \cup \ssm({Y(u)^\circ}) \bigg)\cup \ssm({Y(v)^\circ})
 =c_1(\calL_\lambda) \cup \bigg( \ssm({Y(u)^\circ})\cup \ssm({Y(v)^\circ}) \bigg).
\]
By \Cref{equ:CSMLR}, \Cref{equ:ccoeff} and \Cref{lem:CSMLR},
taking the coefficients of $\ssm({Y(w)^\circ})$ in both sides of the above equation gives
$$c^u_{\lambda,u} d^w_{u,v}+\sum_{\substack{u < x}}
c^x_{\lambda,u} d^w_{x,v}=
c^w_{\lambda,w} d^w_{u,v}+\sum_{\substack{y< w}}
c^w_{\lambda,y} d^y_{u,v}.$$
This proves the desired equality.
\end{proof}
Note that by \Cref{thm:chevalley}, the Chevalley coefficient $c_{\lambda, x}^{y} \neq 0$ only if $x \le y$.
Furthermore, if $u \neq w$, it was proved in \cite[Proposition A.3]{mihalcea:eqqhom} that
one can find $\lambda$ such that $c^w_{\lambda,w} \neq c^u_{\lambda,u}$.
Thus, the identity in \Cref{lem:recLR} may be rewritten as:
\begin{equation}
\label{eq:recLR}
d^w_{u,v}=\frac{ 1 }{ c^w_{\lambda,w} - c^u_{\lambda,u}}\left( \sum_{\substack{ u < x}} 
c^x_{\lambda,u}d^w_{x,v} -\sum_{\substack{ y < w}} c^w_{\lambda,y} d^y_{u,v}\right) \/.
\end{equation}
Then the same argument as the one from \cite{mihalcea:eqqhom} shows that 
\Cref{eq:recLR} gives a recursive relation for the 
coefficients $d_{u,v}^w$. We briefly recall the salient points. 
If $u=w$, then the coefficient $d_{u,v}^u$ is known from 
\Cref{lem:CSMLR} and \Cref{prop:loc}. If
$u \neq w$, we may choose $\lambda$ (depending on $u,w$) such that
$c^w_{\lambda,w} - c^u_{\lambda,u} \neq 0$. The
coefficients $d_{x,v}^w$ and $d_{u,v}^y$ on the right hand side of
\Cref{eq:recLR} satisfy $\ell(w)-\ell(x), \ell(y) - \ell(u) < \ell(w)- \ell(u)$.
To conclude, an inductive argument on $u \le w$ and $\ell(w) - \ell(u)$ shows that:
\begin{corol}\label{cor:SMLR} The coefficients $d_{u,v}^w$ are algorithmically determined
by the following:
\begin{itemize}
\item $d_{u,v}^w =0$ for $u \nleq w$;

\item $d_{u,v}^u$ (known from \Cref{lem:CSMLR} and \Cref{prop:loc});

\item For $u < w$, the coefficients $d_{u,v}^w$ are determined recursively from
\Cref{eq:recLR}, in terms of coefficients
$d_{u',v}^{w'}$ where $u \le u'$, $v' \le w$ and $\ell(v') - \ell(u') < \ell(w) - \ell(u)$. 
\end{itemize}
\end{corol}
An algorithm to calculate $d_{u,v}^u$, in terms of paths on a decorated version of the Bruhat graph of
$G/P$, will be discussed in \S \ref{sec:LambdaBruhat} below.

\section{Equivariant multiplicities} \label{sec:equimult}

\subsection{Definition and a smoothness criterion} Let $X$ be a scheme with a torus action $T$, and $p \in X$ a
torus fixed point. Following \cite[\S 4]{brion:eq-chow}, we say that $p$ is non-degenerate if $0$ is not a weight
of the (Zariski) tangent space $T_p X$. The set of weights of $T_p X$ will be called the weights of $p$. 
In \cite[\S 4.2]{brion:eq-chow}, Brion defined an equivariant 
multiplicity $e_p(\kappa)$ for a (Borel-Moore) equivariant homology class 
$\kappa \in H_*^T(X)$ at a non-degenerate torus fixed point $p$. The goal of this section is to prove 
\Cref{thm:csmsmooth} which provides a smoothness criterion utilizing equivariant multiplicities and CSM
classes; it generalizes a similar criterion
by Brion, who utilizes fundamental classes. This will be needed to identify the right hand side of the \Cref{intro:SMfrac} 
with an equivariant multiplicity.

We fix some notation. For a $T$-module $V = \bigoplus_{i=1}^n \C_{\chi_i}$ denote by $\Phi(V) = \{ \chi_i : 1 \le i \le n\}$ the set of its weights. The {\em Chern class} of $V$ is defined by $c^T(V) := \prod_{i=1}^n (1+ \chi_i)$. 
The {\em Euler class} of $V$ is defined by $e^T(V) := \prod_{i=1}^n \chi_i$; these 
as elements of $H^*_T(pt)$. 

Next we recall the definition of the equivariant multiplicity from \cite[\S 4.2]{brion:eq-chow}.
Recall that $H^*_T(pt)_{\loc} $ denotes the fraction field of $H^*_T(pt)$. 

\begin{theorem}[Brion]\label{thm:def-eq-m}  Let $p \in X$ be a non-degenerate fixed point and let $\chi_1, 
\ldots, \chi_n$ be its weights. 

(a) There exists a unique $H^*_T(pt)$-linear map $e_{p,X}: H_*^T(X) \to H^*_T(pt)_{\loc}$ such that 
\[e_{p,X}[p]_T =1 \textrm{ and } e_{p,X}[Y]_T = 0 \]
for any $T$-invariant subvariety $Y \subset X$ with $p \notin Y$. Furthermore, the image $Im(e_{p,X})$ is 
contained in {$\frac{1}{\chi_1 \cdot \ldots \cdot \chi_n} H^*_T(pt)$}.

(b) For any $T$-invariant subvariety $Y \subset X$, the rational function $e_{p,X}[Y]_T$
is homogeneous of degree $-\dim Y$ and it coincides with $e_{p,Y}[Y]_T$.

(c) The point $p$ is non-singular in $X$ if and only if
\[ e_{p,X}[X]_T =\frac{1}{\chi_1 \cdot \ldots \cdot \chi_n} \/. \]

\end{theorem}

\begin{corol}\label{cor:eqmultsmooth} Let $X$ be a scheme with a torus action $T$, and let 
$p \in X$ be a non-singular, non-degenerate, torus fixed point. Then the equivariant multiplicity equals
\[ e_{p,X}(\kappa) = \frac{\kappa|_p}{e^T (T_p X) } \quad \forall \kappa \in H^*_T(X) \/. \]
\end{corol}
\begin{proof} Let $\iota_p : \{ p \} \to X$ denote the embedding. We may 
may find a $T$-invariant neighborhood $V$ of $x$ such that $V$ is smooth.
(For instance, $V = X \setminus X_{sing}$, the complement of the singular locus of $X$.)
Using the terminology from \cite{fulton:IT}, the morphism $\iota_p$ is an l.c.i. morphism (it is the 
composition of the regular embedding $\{ p \} \hookrightarrow V$ and the open embedding 
$V \subset X$). 
Then by the self intersection formula $[p]|_p= \iota_p^* (\iota_p)_*[p]_T = e^T (T_p X)[p]_T$ in $H_*^T(pt)$.
When regarded in the equivariant cohomology, this shows that $[p]|_p/e^T (T_p X) = 1$.
Then the statement follows from the part (a) of \Cref{thm:def-eq-m}.
\end{proof}

We will need the following homological analogue of \cite[Lemma 9.3]{AMSS:motivic}.
\begin{lemma}\label{lemma:locssm} Let $X \subset M$ be a $T$-equivariant closed embedding of 
irreducible $T$-varieties. Let $p \in X$ be a $T$-fixed point, non-degenerate and smooth in $X$ and $M$.
Denote by $N_{p,X} M$ the normal space
of $X$ in $M$ at $p$. Then:
\[ \iota^*_p \csm(X) = c^T(T_pX) \cdot e^T(N_{p,X} M)~\/. \] 
\end{lemma}
\begin{proof} Since the arguments essentially mimic those from \cite{AMSS:motivic},
we will be brief. Again we may find a $T$-invariant neighborhood $x \in V \subset M$ such that $V \cap X$ is smooth.
Let $j:V \to M$ be the inclusion. This is an open embedding, thus smooth, with trivial relative tangent bundle. Factor the inclusion $\iota_p$ as 
$\xymatrix{ \{ p \} \ar[r]^{\iota_p'} & V \ar[r]^j & M }$. Note that the restriction $i: X \cap V \subset V$ is proper, by base-change.  
We apply the Verdier--Riemann--Roch (see e.g. \cite{AMSS:shadows}), to obtain
\[ \iota_p^* \csm(X) = (\iota'_p)^* j^* \csm(X) = (\iota'_p)^* i_*\csm(X \cap V) \/. \]
Since $X \cap V$ is smooth, it follows that $i_*\csm(X \cap V) = i_* (c^T (T({X \cap V})) \cap [X \cap V]_T) \in H_*^T(X \cap V)$.
By the self-intersection formula,
\[ (\iota'_p)^* i_*\csm(X \cap V) = c^T (T_{p}X) \cdot e^T(N_{p,X \cap V}V) = c^T (T_{p}X) \cdot e^T(N_{p,X}M ) \/, \]
as claimed.\end{proof}
The following is immediate from the previous Lemma.
\begin{corol}\label{cor:issm} Let $M$ be a smooth $T$-variety, and assume all the other hypotheses from \Cref{lemma:locssm}. Then:
\[ \iota_p^*(\ssm(X)) = \frac{e^T(N_{p,X} M)}{c^T(N_{p,X} M)} \/. \]
\end{corol}

\begin{theorem}\label{thm:csmsmooth} Let $X$ be an irreducible variety with a $T$-action. Let $p \in X$ be a non-degenerate $T$-fixed point with weights $\chi_1, \ldots, \chi_n$. Then $X$ is smooth at $p$ if and only
if 
\[ e_{p,X}(\csm(X)) = \prod_{i=1}^n  \bigl(1+ \frac{1}{\chi_i} \bigr) \/.\]
\end{theorem}
\begin{proof}If $X$ is smooth at $p$, the claim follows from \Cref{cor:eqmultsmooth} and \Cref{lemma:locssm} (for $X=M$). 
We now prove the converse. From the definition of the CSM classes it follows that
in $H_*^T(X)$,
\begin{equation}\label{E:csmexp} \csm(X) = [X]_T + \sum a_i [V_i]_T \/, \end{equation}
where $a_i \in H^*_T(pt)$ and $V_i \subset X$ are closed irreducible subvarieties such that the terms 
$a_i [V_i]_T \in H_{2j}^T(X)$ for $j < \dim X$.
Indeed, take a Whitney stratification $X = \bigcup X_i$. We may find equivariant desingularizations
$\pi_i: \widetilde{X}_i \to \overline{X_i}$ of the closures of $X_i$ such that $\pi_i$ is an isomorphism
over $X_i$ and $\widetilde{X}_i \setminus \pi_i^{-1}(X_i)$ is a simple normal crossing divisor. Then
by additivity and functoriality, 
\[ \csm(X) = \sum (\pi_i)_*(c^T ( T \widetilde{X}_i)) - (\pi_i)_*(\csm( \widetilde{X}_i \setminus \pi_i^{-1}(X_i))) \/. \]
It is not difficult to check that in this expression, the only term in $H_{2 \dim X}^T(X)$ is $[X]_T$.

From \Cref{E:csmexp} it follows that the leading term (i.e.~the term of lowest degree) in the localization $\iota_p^* (\csm(X))$ equals $\iota_p^*[X]_T$. Since the equivariant multiplicity is $H^*_T(pt)$-linear, we deduce that 
\[ e_{p,X}(\csm(X)) = e_{p,X}([X]_T) + \sum a_i e_{p,X}([V_i]_T) \/. \] By part (b) of \Cref{thm:def-eq-m}, $\deg e_{p,X}([X]_T) 
= - \dim X$ and $\deg e_{p,V_i}([V_i]_T) = - \dim V_i $. Then the leading term of $e_{p,X}(\csm(X))$ has degree 
$-\dim X$, and it must be equal to $e_{p,X}([X]_T)$.~The hypothesis implies that 
$e_{p,X}([X]_T)= \frac{1}{\chi_1 \cdot \ldots \cdot \chi_n}$, and by Brion's criterion from \Cref{thm:def-eq-m}(c), $X$ is smooth at $p$.
\end{proof}

\section{Equivariant multiplicities of CSM classes and smoothness of Richardson varieties}
Next we apply the results in \S\ref{sec:equimult} to prove a smoothness criterion for Richardson varieties in terms of the equivariant multiplicities of their CSM class; cf.~\Cref{thm:smoothP}. This will be used to interpret certain terms in
the hook formula from \S\ref{sec:hook}.

\subsection{Weights of tangent and normal spaces of Schubert varieties}\label{ss:weights} We recall, and also define, various tangent and normal spaces of Schubert varieties which we utilize below. 

Fix $P$ a parabolic subgroup and $w \in W^P$. The fixed point $wP$ is isolated and the tangent space $T_w G/P$ has weights
$\Phi(T_w G/P) := \{ - w(\alpha) : \alpha \in R^+ \setminus R_P^+\}$.  
We consider the following $T$-submodules of $T_w G/P$.

\begin{itemize} \item The tangent spaces $T_w (X(w))$, respectively $T_w(Y(w))$, of $X(w)$ and $Y(w)$, at the smooth point $wP$. They satisfy $T_w X(w) \oplus T_w Y(w) = T_w G/P$ and have weights
\[ \Phi(T_w Y(w)) = \{ - w (\alpha) \in \Phi (T_w G/P): w(\alpha) > 0 \} \/; \] 
\[ \Phi(T_w X(w)) = \Phi(T_w G/P) \setminus \Phi(T_w Y(w)) = \{ - w (\alpha) \in \Phi (T_w G/P): w(\alpha) < 0 \} \/. \]
Note that the condition $w(\alpha)>0$ (respectively $w(\alpha)<0$) is equivalent to $w s_\alpha > w$ (respectively $ws_\alpha < w$). Equivalently, $\Phi(T_w X(w))=S(w)$ from \Cref{equ:wvinversion}.
\item The normal spaces $N_{w,X(w)}G/P$, respectively $N_{w,Y(w)}G/P$, of $X(w)$ and $Y(w)$, at the smooth point $wP$. Note that:
\[ \Phi(N_{w,X(w)}G/P) =\Phi(T_w Y(w))\/; \quad \Phi(N_{w,Y(w)} G/P) =\Phi(T_w X(w)) \/. \]

\item More generally, let $v \le w$ be two elements in $W^P$. Define the $T$-submodule $\tilde{T}_w Y(v)\subset T_w G/P$ by the requirement that its weights are:
\[ \Phi(\tilde{T}_w Y(v)) = \{ -w(\alpha) \in \Phi(T_w G/P): v \leq w s_\alpha \} \/. \]
The space $\tilde{T}_w Y(v)$ is in general not equal to the Zariski tangent space $T_w Y(v)$ of $Y(v)$ at $wP$.
However, if $wP$ is non-singular in $Y(v)$ then $\tilde{T}_w Y(v)$ is the actual tangent 
space. The latter will be the case for most of our applications. We will not need it, but observe
that in general $\tilde{T}_w Y(v)$ is always included in $T_w Y(v)$, see e.g.~\cite{carrell:bruhat}
or \cite[Prop.~12.1.7]{kumar:book}.

\item As before, let $v \le w$ be two elements in $W^P$. Define the $T$-submodule 
$\tilde{N}_{w, Y(v)}G/P \subset T_w G/P$ by the requirement that its weights are:
\[ \Phi(\tilde{N}_{w,Y(v)} G/P) = \{ -w(\alpha) \in \Phi(T_w G/P): v \nleq w s_\alpha \}=\{ \beta \in \Phi(T_w G/P): v \nleq s_\beta w  \} \/, \]
where the second equality follows from the change of variable $\beta=-w(\alpha)$.
Again we observe that if $wP$ is non-singular in $Y(v)$ then this is the genuine normal
space $N_{w, Y(v)}G/P$ of $Y(v)$ at $wP$. Also observe that by definition,
\[ \tilde{T}_w Y(v) \oplus \tilde{N}_{w,Y(v)} G/P = T_w G/P \/. \]
\end{itemize}
We also recall a smoothness criterion for Schubert varieties due to S. Kumar, see also  \cite[page 255 (K)]{brion:eq-chow} for another proof using equivariant multiplicities.
\begin{theorem}[\cite{kumar1996nil} ]\label{thm:kumar} 
Let $v,w \in W^P$ be two Weyl group elements such that $v \le w$. Then the Schubert variety 
$Y(v)\subset G/P$ is smooth at $wP$ if and only if the localization of the 
equivariant fundamental class $ [Y(v)]\in H_T^*(G/P)$ is given by:
 \[ [Y(v)]|_{w}= \prod_{\beta\in  \Phi(\tilde{N}_{w,Y(v)} G/P)}\beta \/.\] 
\end{theorem}
\begin{proof} In \cite{kumar1996nil}, the criterion is stated for $G/B$, but it immediately 
implies the result in $G/P$. For completeness, we indicate the main points. We use the same convention as in the proof of \Cref{prop:loc}.
Let $p:G/B \to G/P$ be the natural projection, and let $Y(v)_B:=\overline{B^-vB/B}\subset G/B$ be the Schubert variety in $G/B$. Since $p$ is a smooth morphism, and if
$v \in W^P$ then $p^{-1}(Y(v)) = Y(v)_B$. In particular, if $v \le w$ are elements in $W^P$, 
then $Y(v)$ is smooth at $wP$ if and only if $Y(v)_B$ is smooth at $wB$. Furthermore, 
 the localization $[Y(vW_P)]|_{w} = [Y(v)_B]|_{wB}$. This finishes the proof.
 \end{proof}

\subsection{Equivariant multiplicities of CSM classes of Richardson varieties} In this section we calculate 
equivariant multiplicities of Richarsdon varieties and their CSM classes. These will show again 
as factors in the generalized hook formula.

Let $v \le w$ in $W^P$ and consider the Richardson variety $R_w^{v}:= X(w) \cap Y(v )\subset G/P$. The $T$-fixed point $wP$ is an isolated $T$-fixed point in $R_w^v$,and the torus weights of the tangent space of $R_w^v$ at $wP$ are non-zero and distinct, therefore $wP$ is non-degenerate in $R_w^v$. We need the following Lemma.

\begin{lemma}\label{lemma:prodloc} Let $w$ be an element in $W^P$. Then 
\[ \csm(X(w))|_w \cdot \ssm(Y(w))|_w = e^T(T_w G/P) \/. \]
\end{lemma}
\begin{proof} From \Cref{lemma:locssm} we obtain that
\[ \begin{split} \csm(X(w))|_w \cdot \ssm(Y(w))|_w &=  
\frac{c^T( T_w X(w)) e^T (N_{w,X(w)} G/P) c^T( T_w Y(w)) e^T (N_{w,Y(w)} G/P)}{c^T (T_w G/P)} \\ & 
= e^T (T_w G/P) \/. \end{split} \] 
Here we utilized that $T_w X(w) \oplus T_w Y(w) = T_w G/P$ since the intersection $X(w) \cap Y(w)$
is proper and transversal and it consists of the single point $wP$.\end{proof}

\begin{prop}\label{prop:eqmult} Let $v \le w$ in $W^P$. Then the following equality holds:
\[e_{w,G/P} (\csm(R_w^{v}))= \frac{\ssm(Y(v))|_w}{\ssm(Y(w))|_w} \/. \]
\end{prop}
\begin{proof} From \Cref{cor:eqmultsmooth} we deduce that 
\[ e_{w,G/P} (\csm(R_w^{v}))= \frac{\csm(R_w^v)|_w}{e^T (T_w G/P)} \/. \]
By Sch{\"u}rmann's transversality formula \cite{schurmann:transversality,AMSS:shadows}
it follows that 
\[\csm(R_w^v)|_w = (\csm(X(w)) \ssm(Y(v)))|_w 
= \csm(X(w))|_w \cdot \ssm(Y(v))|_w \/. \]
Then the claim follows from \Cref{lemma:prodloc} after combining the two equations. 
\end{proof}

Next we utilize {Kumar's smoothness criterion} 
to relate smoothness of Schubert and Richardson varieties in terms of localizations
of CSM and SM classes.

\begin{theorem}\label{thm:smoothP} Let $v,w \in W^P$ be two elements such that $v \le w$.
Then the following are equivalent:
\begin{enumerate}
\item[(a)] The Schubert variety $Y(v)$ is smooth at $wP$;
\item[(b)] The Richardson variety $R_w^v$ is smooth at $wP$;
\item[(c)] The localization of $\csm(Y(v))$ at $wP$ satisfies:
\[ \csm(Y(v))|_{w}= e^T(\tilde{N}_{w,Y(v)}G/P) \cdot c^T(\tilde{T}_{w}Y(v)) \/;\]

\item[(d)] The localization of $\ssm(Y(v))$ at $wP$ satisfies:
\[ \ssm(Y(v))|_{w}= \frac{e^T(\tilde{N}_{w,Y(v)}G/P)}{c^T(\tilde{N}_{w,Y(v)}G/P)} = \prod_{\beta>0, v\nleq s_\beta w} \frac{\beta}{1+\beta}
;\]

\item[(e)] The equivariant multiplicity of $\csm(R_w^v)$ at $wP$ satisfies:
\[e_{w,G/P} (\csm(R_w^{v}))= \prod_{\beta\in S(w/v)} (1+\frac{1}{\beta}),\]
where $S(w/v)$ is defined in \Cref{equ:wvinversion}.
\end{enumerate}
\end{theorem}
\begin{proof} The equivalence $(a) \Leftrightarrow (b)$ follows from Kleiman transversality theorem; see e.g. \cite[Cor.~2.9]{billey.coskun:singularities} for a proof. The second equality in $(d)$, and the equivalence $(c) \Leftrightarrow (d)$,
follow from the definition of the spaces involved (cf.~\S \ref{ss:weights}) and the definition of CSM and SM classes. The equivalence $(d) \Leftrightarrow (e)$ follows from \Cref{prop:eqmult}, taking into account that $Y(w)$ is smooth at $wP$, and using
the formula for $\ssm(Y(w))|_w$ from \Cref{cor:issm}. 

To finish the proof it suffices to show the equivalence $(a) \Leftrightarrow (c)$. The direction $(a) \Rightarrow (c)$ 
follows from \Cref{lemma:locssm}. For the reverse direction, observe that \Cref{E:csmexp} implies that the term 
of lowest degree of $\csm(Y(v))|_w$ is the localization $[Y(v)]|_w$. Therefore, the hypothesis implies that 
\[ [Y(v)]|_w = e^T(\tilde{N}_{w,Y(v)}G/P)\/.\]
Then the claim follows from Kumar's smoothness criterion \cite{kumar1996nil}; see \Cref{thm:kumar}.
\end{proof}

\section{The $\Lambda$-Bruhat graph}\label{sec:LambdaBruhat} In this section we introduce the main combinatorial object in this paper: a directed graph depending on an `admissible function' $\Lambda$ assigning weights to vertices, and whose sums over weighted paths give algorithms to calculate 
localization of SM classes of Schubert cells and varieties; cf.~ \Cref{prop:d_P} and \Cref{thm:LocCSM2}.
Similar graphs based on Chevalley-type recursions have been used in \cite{mihalcea:eqqgr,mihalcea:eqqhom,naruse2014schubert,naruse2019skew} to provide algorithms for Schubert multiplication in the equivariant quantum cohomology of flag manifolds. In the next section we will utilize this graph to formulate and prove a generalization of the Nakada's colored formula. 

For a parabolic subgroup $P$ recall that $X^*(T)_P$ denotes the set of integral weights orthogonal 
to roots in $R_P^+$. We recall the following
characterization of the covering relations in the Bruhat order in $G/P$; see e.g. \cite[Lemma 4.1]{fulton.woodward}.
\begin{lemma}\label{lem:P}
Let $x \neq y$ be two elements in $W^P$. Then the following are equivalent:

\begin{itemize} \item[(a)] There exists $\gamma \in R^{+} \setminus R_P^+$ such that $x s_\gamma  W_P=y W_P$.

\item[(b)] There exists $\beta \in R^{+} \setminus R_P^+$ such that $s_\beta x W_P=y W_P$.

\end{itemize} Furthermore, $\beta$ and $\gamma $ are unique with these properties.

\end{lemma} 

\begin{proof} The equivalence of (a) and (b), and the uniqueness of $\beta$ are proved in 
\cite[Lemma 4.1]{fulton.woodward}. This lemma also shows that $s_\gamma$ 
is unique up to a conjugation by an element in $W_P$. If 
$s_\gamma W_P = s_{\gamma'}W_P$ then $\gamma = \gamma'$ by 
\cite[Lemma 2.2]{buch.m:nbhds}.\end{proof}
\begin{remark}\label{rem:gammabeta} When we analyze edges of the Bruhat graph below, 
we need to consider situations where
$x W_P = y s_{\gamma'} W_P < y W_P= x s_\gamma W_P$ for $x,y \in W^P$. As observed in the proof 
above $s_{\gamma'}$ is a $W_P$ conjugate of $s_\gamma$. Furthermore, if
$xW_P < s_\beta x W_P = x s_\gamma W_P$ then $s_\beta x > x$, $\gamma = x^{-1}(\beta)$,
and for any weight $\lambda$ such that $Stab_W(\lambda) = W_P$,
\[ \langle \lambda , \gamma^\vee \rangle = \langle \lambda, (\gamma')^\vee \rangle = \langle x(\lambda), \beta^\vee \rangle 
\quad \/. \]
\end{remark}
For $v<w\in W^P$, recall that $[v, w]^P:=\{x\in W^P\mid v\leq x\leq w\}$ and set $[v,w)^P=[v, w]^P\setminus \{ w \}$.
\begin{defin}[Admissible function] Let $v<w\in W^P$. An \textit{admissible function},
$\Lambda=\Lambda_{v,w}:[v, w]^P \to X^{*}(T)_P$ is any assignment $x \mapsto \lambda_x$ 
such that $x(\lambda_x)\neq w(\lambda_x)$ for all $x\in [v, w)^P$.
\end{defin} Admissible functions always exist: one example is given by 
\begin{equation}\label{E:standardLa} \Lambda(x)=\varpi_P:= \displaystyle\sum_{\alpha_i\in \Sigma \setminus \Sigma_P}\varpi_i \/; \quad 
\forall x\in [v,w]^P \/. \end{equation} (This follows because $w(\varpi_P) = \varpi_P$ if and only if $w \in W_P$; cf.~\cite[Ch.5, \S 4.6]{bourbaki:Lie4-6}.) 
We call this the {\em standard} admissible function. {Another important situation is when $v<w$ are $\pi$-minuscule
and $P$ corresponds to the stabilizer of $\pi$, i.e., $W_P = W_\pi$. Then the constant function $\Lambda \equiv \pi$ is
admissible.}
%Then from \Cref{lem:gamma} it follows that $\varpi_P = \pi$.} 
Admissible functions appeared in 
\cite[\S 7]{mihalcea:eqqhom} in the study of the equivariant quantum 
cohomology ring of flag manifolds.

\begin{defin}[$\Lambda$-Bruhat graph]\label{def:ref_di} Let $v,w \in W^P$ such that $v <w$, and let $\Lambda: [v,w]^P \to X^*(T)_P$ be an admissible function. 
To this data we associate a decorated directed graph $\Gamma=(V,E)$ and two functions, 
$\mathcal{W}_\Lambda: V \to {X^*(T)}$
and $m_\Lambda: E \to \mathbb{Z}_+$, 
as follows:
\begin{enumerate} \item The set of vertices is defined by $V=[v,w]^P$.
\item There is an oriented edge $x{\to} y$ whenever 
$xW_P < yW_P$ and
$yW_P=x s_\gamma W_P$.
\item Each vertex $x \in V$ is decorated by a weight
\[ \mathcal{W}_\Lambda(x) := x(\lambda_x) - w(\lambda_x) \/. \]  
By the definition of $\Lambda$, if $x \neq w$, the weight $\mathcal{W}_\Lambda(x) \neq 0$.
We will call $\mathcal{W}_\Lambda: V \to X^*(T)$ the 
{\bf $\Lambda$-weight function} of the graph.
\item Each edge $xW_P {\to} y W_P=x s_\gamma W_P$ is decorated by the 
{\bf Chevalley multiplicity}: %defined by
\[ m_\Lambda(x,y) := \langle \lambda_x, \gamma^\vee \rangle \/. \]
\end{enumerate}
We will refer to this graph as the {\bf $\Lambda$-Bruhat graph} determined by the triple $(v,w,\Lambda)$.
If $\Lambda \equiv \pi$ is a constant admissible function, then we set 
$\mathcal{W}_\pi = \mathcal{W}_\Lambda$, $m_\pi = m_\Lambda$.  
\end{defin}
If one ignores the orientation and the admissible function $\Lambda$, then the $\Lambda$-Bruhat graph is the 
$1$-skeleton of the $T$-action on $G/P$. This is the graph utilized in the GKM theory, and to calculate curve neighborhoods of 
Schubert varieties \cite{buch.m:nbhds}; it is also called the Bruhat graph. 
It contains the (unoriented) quantum 
Bruhat graph from \cite{BFP99} and \cite{LNSSS}, and it is related to 
``Games with hook structure'' defined by Kawanaka \cite{Kaw15}.

Next we record the following lemma.
\begin{lemma}\label{maxP} (a) Let $(v,w,\Lambda)$ be a  $\Lambda$-Bruhat graph and let $(x,y)$ be an edge such that
$yW_P = s_\beta x W_P = x s_\gamma W_P$. If $\Lambda(x)=\Lambda(y)=\lambda$ then
$x(\lambda) - y (\lambda) = m_\Lambda(x,y) \beta$.

(b) Let $\pi \in X(T)_P$ be a dominant integral weight and assume we are given a constant admissible function $\calW_\pi \equiv \pi$.
Then for any edge $x\overset{\beta}{\to} y$ as in (a) (with the notation from~\Cref{defin:order}): 
\begin{equation}\label{E:Wla} 
\mathcal{W}_\pi(x) = \mathcal{W}_\pi(y) + m_\pi(x,y) \beta \/.
\end{equation}
 \end{lemma}
\begin{proof} 
By definition,
\begin{align*}
x(\lambda) - y (\lambda)=& x(\lambda)-xs_\gamma(\lambda)
=\langle \lambda,\gamma^\vee\rangle x(\gamma)
=m_\Lambda(x,y) \beta,
\end{align*}
where the last equality follows from \Cref{rem:gammabeta}. This proves part (a). Part (b) follows form (a) and 
the definition of $\calW_\pi$.\end{proof}

In the next section we will analyze $\Lambda$-Bruhat graphs for minuscule elements, and we will need the following 
result.
\begin{prop}\label{cor:multweight} Let $v < w$ be $\pi$-minuscule elements, and let $P$ be the parabolic 
subgroup defined by $W_P = Stab_W(\pi)$. Consider the 
$\Lambda$-Bruhat graph $(v,w,\Lambda)$ with 
the constant admissible function $\Lambda \equiv \pi$. Then the following hold:

(1) Let $x,y \in [v,w]^P$ be two elements and let $x \to y$ be an edge such that
$x W_P < y W_P = x s_\gamma W_P$. Then the Chevalley multiplicity 
\[ m_\pi(x,y) = \langle \pi, \gamma^\vee \rangle =1 \/. \] 

(2) The $\Lambda\equiv \pi$-weight of the vertex $x$ equals 
\[\mathcal{W}_\pi(x) = x(\pi) - w(\pi) = 
\sum_{i=0}^s \beta_i \/, \]
where 
$x=x_0\overset{\beta_1}{\to} x_{1}\to \ldots \to x_{s-1}\overset{\beta_{s}}{\to} x_s =w$ 
is any chain in the $\Lambda$-Bruhat graph.

(3) The admissible function $\calW_\pi$ is injective.
\end{prop}
\begin{proof} By \Cref{lem:wminuscule}(2), since $w$ is $\pi$-minuscule, each element in 
$[v,w]^P$ is again 
$\pi$-minuscule. Let $\gamma'$ be the positive root such that $xW_P = y s_{\gamma'}W_P$ (cf.~\Cref{lemma:ubetav}).
From \Cref{rem:gammabeta}, the multiplicity
$m_\pi(x,y) = \langle \pi, (\gamma')^\vee \rangle$. Since $y s_{\gamma'} W_P < y W_P$ 
implies that $y s_{\gamma'} < y$, and since $y$ is $\pi$-minuscule,
$\langle \pi, (\gamma')^\vee \rangle = 1$ from \Cref{lem:gamma}. This proves part (1). 

From \Cref{maxP}(b) it follows that for any chain as in the hypothesis,
\[\calW_\pi(x) = x(\pi) - w(\pi) = \sum m_\pi(x_{i-1}, x_{i}) \beta_i \/. \]
Since by part (1) all multiplicities $m_\pi(x_{i-1}, x_{i})=1$, part (2) follows.
Finally, if $\calW_\pi(x) = \calW_\pi(y)$ then $x(\pi) = y(\pi)$, thus $x=y$ in $W^P$,
proving part (3).
\end{proof}

We provide examples of $\Lambda$-Bruhat graphs below; in these examples we will remove those edges of
Chevalley multiplicity $0$, as they will not contribute to our algorithm for the SM structure constants.
We will encode the
function $\mathcal{W}_\Lambda$ by representing a vertex $x \neq w$ as $\frac{x}{\mathcal{W}_\Lambda(x)}$.
\begin{example}\label{ex:A2} Consider $G=\mathrm{SL}_3(\mathbb{C})$ and $P=B$. In this case
$G/B= \mathrm{Fl}(3)$, the flag variety parametrizing complete flags $(F_1 \subset F_2 \subset \mathbb{C}^3)$. 
Set $v:=id$ and $w:=w_0 = s_1 s_2 s_1$ and consider two functions $\Lambda_i: W \to X^*(T)^+$ defined by:
\[ \Lambda_1(x) = \begin{cases} \varpi_1 & \textrm{if } x \neq s_2 s_1 \/; \\ \varpi_2 & \textrm{if } x = s_2 s_1 \/. \end{cases}
\quad \quad 
\Lambda_2(x) \equiv n_1 \varpi_1 + n_2 \varpi_2 \/, \quad (n_1, n_2 >0) \/.\]
One may easily check that both are admissible functions. The corresponding graphs for the triple $(id,w_0,\Lambda_i)$ are below, from left to right;
the edges are decorated with their Chevalley multiplicities, and a vertex $u \neq w$ is replaced by its decoration $\frac{u}{\mathcal{W}_{\Lambda_i}(u)}$.

%\begin{minipage}{8cm}
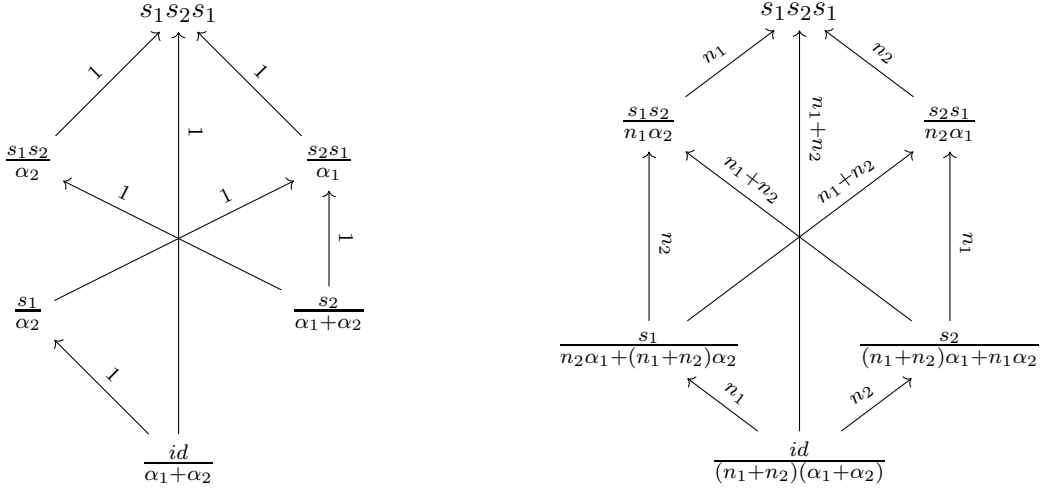
\begin{tikzpicture}

\node (v0) at (3,6){$s_1 s_2 s_1$};
\node (v1) at (1,4){$\frac{s_1 s_2}{\alpha_2}$};
\node (v2) at (5,4){$\frac{s_2 s_1}{\alpha_1}$};
\node (v3) at (1,2){$\frac{s_1}{\alpha_2}$};
\node (v4) at (5,2){$\frac{s_2}{\alpha_1+\alpha_2}$};
\node (v5) at (3,0){$\frac{id}{\alpha_1+\alpha_2}$};

\draw[<-] (v0)--(v1) node [pos=0.50,  above, sloped] {$_{1}$};
\draw[<-] (v0)--(v2) node [pos=0.50,  above, sloped] {$_{1}$};
\draw[<-] (v0)--(v5) node [pos=0.25,  above, sloped] {$_{1}$};
\draw[<-] (v1)--(v4) node [pos=0.25,  above, sloped] {$_{1}$};
\draw[<-] (v2)--(v3) node [pos=0.25,  above, sloped] {$_{1}$};
\draw[<-] (v2)--(v4) node [pos=0.50,  above, sloped] {$_{1}$};
\draw[<-] (v3)--(v5) node [pos=0.50,  above, sloped] {$_{1}$};
\end{tikzpicture}
\hspace{2cm}
\begin{tikzpicture}

%\hspace{8cm}
\node (v0) at (3,6){$s_1 s_2 s_1$};
\node (v1) at (1,4.5){$\frac{s_1 s_2}{n_1 \alpha_2}$};
\node (v2) at (5,4.5){$\frac{s_2 s_1}{n_2 \alpha_1}$};
\node (v3) at (1,1.5){$\frac{s_1}{n_2 \alpha_1 + (n_1 + n_2) \alpha_2}$};
\node (v4) at (5,1.5){$\frac{s_2}{(n_1+n_2)\alpha_1+n_1 \alpha_2}$};
\node (v5) at (3,0){$\frac{id}{(n_1+n_2)(\alpha_1 +\alpha_2)}$};

\draw[<-] (v2) -- (v3) node [pos=0.25,  above, sloped] {$_{n_1+n_2}$};

\draw[<-] (v1) -- (v4) node [pos=0.25, above, sloped] {$_{n_1+n_2}$};

\draw[<-] (v0) -- (v5) node [pos=0.25, above, sloped] {$_{n_1+n_2}$} ;

\draw[<-] (v0) -- (v1) node [pos=0.55,  above, sloped] {$_{n_1}$} ;

\draw[<-] (v0) -- (v2)  node [pos=0.55,  above, sloped] {$_{n_2}$};

\draw[<-] (v3) -- (v5)  node [pos=0.55,  above, sloped] {$_{n_1}$} ;

\draw[<-] (v4) -- (v5)  node [pos=0.55,  above, sloped] {$_{n_2}$} ;

\draw[<-] (v1) -- (v3)  node [pos=0.55, above, sloped] {$_{n_2}$};

\draw[<-] (v2) -- (v4)  node [pos=0.55,  above, sloped] {$_{n_1}$} ;
\end{tikzpicture}
%\end{minipage}
\end{example}

\begin{example}\label{ex:B3} Consider the Lie type $B_3$, i.e.~$G = \mathrm{SO}(7,\mathbb{C})$.
The simple roots are $\Delta = \{ \alpha_1, \alpha_2, \alpha_3 \}$ with $\alpha_3$ short.~We choose 
$P$ to be the maximal parabolic determined by the set $\Delta_P = \{ \alpha_1, \alpha_3 \}$. 
Geometrically, $G/P$ is the submaximal isotropic Grassmannian
$\mathrm{IG}(2,7)$ parametrizing subspaces of dimension $2$ which 
are isotropic with respect to the non-degenerate symmetric 
form in $\mathbb{C}^7$. We pick the standard admissible function
$\Lambda(x) \equiv \varpi_2$ and $v=id,w=s_2 s_1 s_3 s_2$. All simple edges have Chevalley 
multiplicity $1$, and the double edges have multiplicity $2$.

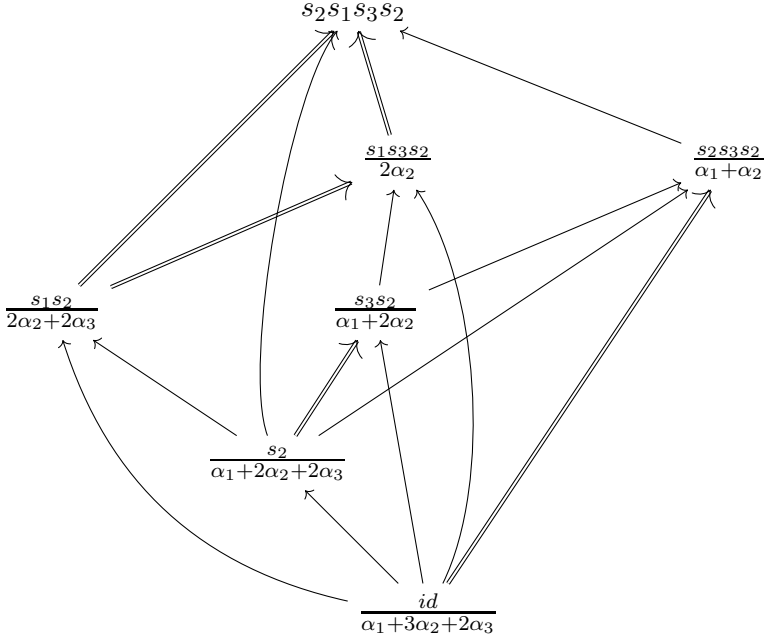
\begin{minipage}{11.5cm}
\begin{tikzpicture}

\node (v0) at (10,8){$s_2 s_1 s_3 s_2$};

\node (v1) at (10.6,6){$\frac{s_1 s_3 s_2}{2 \alpha_2}$};

\node (v2) at (15,6){$\frac{s_2 s_3 s_2}{\alpha_1+\alpha_2}$};

\node (v3) at (6,4){$\frac{s_1 s_2}{2\alpha_2+2 \alpha_3}$};

\node (v4) at (10.3,4){$\frac{s_3 s_2}{\alpha_1+2\alpha_2}$};

\node (v5) at (9,2){$\frac{s_2}{\alpha_1+2\alpha_2+2\alpha_3}$};

\node (v6) at (11,0){$\frac{id}{\alpha_1+3\alpha_2+2\alpha_3}$};

\draw[<-,double distance=0.8pt] (v0)-- node[right]{}(v1);
\draw[<-] (v0)--node[right] {}(v2);
\draw[<-,double distance=0.8pt] (v0)--node[left] {}(v3);

\draw[<-,double distance=0.8pt] (v1)--node[above] {}(v3);
\draw[<-] (v1)--node {}(v4);

\draw[<-] (v2)--node[left] {}(v4);
\draw[<-] (v2)--node {}(v5);
\draw[<-,double distance=0.8pt] (v2)--node[right] {}(v6);

\draw[<-] (v3)--node[right] {}(v5);
\draw[<-,double distance=0.8pt] (v4)--node {}(v5);
\draw[<-] (v5)--node{}(v6);

\draw[bend right,distance=2cm] (v3) to node{}(v6)[<-];
\draw[<-] (v4)--node{}(v6);

\draw[bend right,distance=1cm] (v0) to 
node[below] {${}$} 
(v5) [<-];
\draw[bend left,distance=1.5cm] (v1) to
node[below] {} (v6)[<-];
\end{tikzpicture}
\end{minipage}
\end{example}

We now give a formula for the coefficients $d_{u,w}^w$ from \Cref{equ:CSMLR} in terms of summation 
over weighted paths in the $\Lambda$-Bruhat graph. 
\begin{prop}\label{prop:d_P}
For any $v\leq  w\in W^P$, fix an admissible function $\Lambda:[v,w]^P\to X^{*}(T)_P$, 
and let $\Gamma=(v,w,\Lambda)$ be the associated $\Lambda$-Bruhat graph. Then for $u\in [v,w]^P$,
\begin{equation}\label{eq:d}
d_{u,w}^w=\left(
\sum \frac{m_\Lambda(x_{r},x_{r-1})}{\mathcal{W}_\Lambda(x_{r})} \cdot \frac{m_\Lambda(x_{r-1},x_{r-2})}{\mathcal{W}_\Lambda(x_{r-1})}\cdot \ldots \cdot
\frac{m_\Lambda(x_1,x_0)}{\mathcal{W}_\Lambda(x_1)} \right)d_{w,w}^w .
\end{equation} where the sum is over integers $r \ge 0$, and over all directed paths 
$u= x_r \overset{}{\to} x_{r-1} \overset{}{\to} \ldots \overset{}{\to} x_0 = w$ in $\Gamma$. 
\end{prop}

\begin{proof}
There is nothing to do when $u=w$. If $u<w$, set $\lambda_u:= \Lambda(u)$. Then from
\Cref{eq:recLR},
\[ \begin{split} d_{u,w}^w & = \frac{1}{w(\lambda_u) - u(\lambda_u)} \sum_{u < x} c_{\lambda_u, u}^x d_{x,w}^w
 = \frac{1}{w(\lambda_u) - u(\lambda_u)} \sum_{x= us_\alpha >u ; \alpha >0} 
 c_{\lambda_u, u}^x d_{x,w}^w
\\ & = \frac{1}{w(\lambda_u) - u(\lambda_u)} \sum_{x= us_\alpha >u ; \alpha >0} \langle -\lambda_u , \alpha^\vee \rangle  d_{x,w}^w = 
\frac{1}{\mathcal{W}_\Lambda(u)} \sum_{u \overset{}{\to} x } m_\Lambda(u,x) d_{x,w}^w \\ &
= \sum_{u \overset{}{\to} x } \frac{m_\Lambda(u,x)}{\mathcal{W}_\Lambda(u)} d_{x,w}^w \/. \end{split}
\] Here the second equality follows from \Cref{thm:chevalley} and the definition of $c_{\lambda, u}^x$ in \Cref{equ:ccoeff}, and the rest are from definitions. Then the statement follows
by induction descending from $w$, on those elements $x$ such that $u<x\le w$.
\end{proof}
From \Cref{lem:CSMLR} we deduce another interpretation of the sum in the previous proposition:
\begin{corol}\label{thm:LocCSM2} Let $v \le w$ be two elements in $W^P$. 
With the hypotheses from \Cref{prop:d_P}, the following hold:
\begin{equation}\label{hook-formula}
\frac{\ssm(Y(v))|_w}{\ssm(Y(w))|_w} =
\sum
\frac{m_\Lambda(x_r,x_{r-1})}{\mathcal{W}_\Lambda(x_r)}\cdot \frac{m_\Lambda(x_{r-1},x_{r-2})}{\mathcal{W}_\Lambda(x_{r-1})}\cdot \ldots \cdot \frac{m_\Lambda(x_{1},x_{0})}{\mathcal{W}_\Lambda(x_{1})},
\end{equation} where the sum is over integers $r \ge 0$, and over all directed paths 
$v\leq x_r \overset{}{\to} x_{r-1} \overset{}{\to} \ldots \overset{}{\to} x_0 = w$ in $\Gamma$.
\end{corol}
\begin{proof} By \Cref{lem:CSMLR} and the additivity of CSM classes, 
the left hand side of \eqref{hook-formula} equals to 
\[\displaystyle \frac{\ssm(Y(v))|_w}{\ssm(Y(w))|_w}=\frac{\sum_{v\leq u}\ssm(Y(u)^\circ)|_w}{\ssm(Y(w)^\circ)|_w} =\sum_{v\leq u\leq w} \frac{d_{u,w}^{w}}{d_{w,w}^{w}}
 \/.
\] 
Here we also utilized that $\ssm(Y(u)^\circ)|_w=0$ if $w \ngeq u$.
Then the claim follows from \Cref{prop:d_P}.\end{proof}

\begin{remark} By \Cref{prop:loc} and \Cref{prop:eqmult}, the left hand side of \Cref{hook-formula} 
is an equivariant multiplicity, which may be calculated explicitly:
\begin{equation}\label{equ:explicit} \frac{\ssm(Y(v))|_w}{\ssm(Y(w))|_w} = 
e_{w,G/P}(\csm(R_w^v)) =\frac{\sum \beta_{j_1}\beta_{j_2}\cdot \ldots \cdot \beta_{j_k}}{\beta_1\cdot \ldots \cdot \beta_\ell} \/;
\end{equation}
here the summation is over $ 1\leq j_1<j_2<\cdots <j_k\leq \ell$ such that 
$vW_P\leq s_{i_{j_1}}s_{i_{j_2}}\cdots s_{i_{j_k}}W_P$ (the notation is the same as in \Cref{prop:loc}). 
As we observed in \Cref{thm:smoothP}, if $Y(v)$ is smooth at $wP$, then 
both the numerator and the denominator of this expression may be written as products.
This is the key observation which leads to a generalization of Nakada's hook formula in the next section.\end{remark}

\section{A generalized colored hook formula}\label{sec:genhook}
In this section we prove the main result of this paper - 
the generalization of Nakada's colored hook formula, together with several corollaries 
of it.
\subsection{The colored hook formula and consequences}\label{sec:col-and-conseq}
Recall that for $v < w \in W$, $S(w/v):=\{\beta\in R^+\mid v\le s_\beta w<w\}$.
\begin{theorem}\label{thm:genNak} Let $v\leq  w\in W^P$, and fix an admissible 
function $\Lambda:[v,w]^P\to X^{*}(T)_P$ with the associated 
$\Lambda$-Bruhat graph $\Gamma=(v,w,\Lambda)$. Then: 
\begin{center}
$Y(v)\subset G/P$ is smooth at $wP\in G/P$  if and only if
\[
\sum\frac{m_\Lambda(x_r,x_{r-1})}{\mathcal{W}_\Lambda(x_r)}\cdot \frac{m_\Lambda(x_{r-1},x_{r-2})}{\mathcal{W}_\Lambda(x_{r-1})}\cdot \ldots \cdot \frac{m_\Lambda(x_{1},x_{0})}{\mathcal{W}_\Lambda(x_{1})}
\;=\;
\prod_{\beta \in S(w/v)} \left(1+\frac{1}{\beta}\right),\]
\end{center}
where the sum is over integers $r \ge 0$, and over all directed paths 
$v\leq x_r \overset{}{\to} x_{r-1} \overset{}{\to} \ldots \overset{}{\to} x_0 = w$ in $\Gamma$.
\end{theorem}

\begin{proof} We proved in \Cref{prop:eqmult} that 
\[ \frac{\ssm(Y(v))|_w}{\ssm(Y(w))|_w}=e_{w,G/P} (\csm(R_w^{v})) \/. \]
Then by 
\Cref{thm:smoothP}, $Y(v)\subset G/P$ is smooth at $wP\in G/P$ if and only if 
\[ \frac{\ssm(Y(v))|_w}{\ssm(Y(w))|_w} = \prod_{\beta\in S(w/v)}(1+\frac{1}{\beta}) \/. \]
Now observe that by \Cref{thm:LocCSM2} the left hand side of this expression is the sum in the statement.\end{proof}
An important particular case of \Cref{thm:genNak} is to consider a constant admissible function.
For example, let $\pi\in X^*(T)$ be any dominant integral weight, and as usual define 
$P$ by $\Stab_W(\pi)=W_P$. Set
$\Lambda(x) \equiv \pi$ for $x \in [v,w]^P$, and recall that $\calW_\pi$ denotes the associated
weight function. Let $m_i = \langle \pi, \gamma_i^\vee \rangle$ be the multiplicity of the 
edge $x_i \to x_{i-1}$ and let $\beta_i$ be the positive root
$\beta_i$ such that $x_{i-1}W_P = s_{\beta_i}x_i W_P$.
With this notation, and by \Cref{thm:genNak}, we deduce the following.
\begin{corol}\label{col:NakP}
Under the above assumptions,
we have the following equivalence:
\begin{center}
$Y(v)\subset G/P$ is smooth at $wP\in G/P$  if and only if
\[
\sum \frac{m_r}{m_1\beta_1+m_2\beta_2+\cdots +m_r\beta_r} \cdot \ldots \cdot 
\frac{m_1}{m_1\beta_1}
=\prod_{\beta \in S(w/v)}\left(1+\frac{1}{\beta}\right),
\]
\end{center}
where the sum is over all integers $r \ge 0$, and over all directed paths 
$v\leq x_r \overset{\beta_r}{\to} x_{r-1} \overset{\beta_{r-1}}{\to} \ldots \overset{\beta_1}{\to} x_0 = w$ in $\Gamma= (v,w,\pi)$. 
\end{corol}
\begin{proof} Immediate from \Cref{thm:genNak}, since $\mathcal{W}_\pi(x_k)=\sum_{i=1}^k m_i \beta_i$,
by \Cref{E:Wla}. 
\end{proof}

\begin{corol}\label{col:Nak} 
Let $v \le w \in W$ be two $\pi$-minuscule elements for a dominant integral weight $\pi$,
and $P\subset G$ be the parabolic subgroup satisfying 
$\Stab_W(\pi)=W_P$. Then: 
\begin{center}
$Y(v)\subset G/P$ is smooth at $wP\in G/P$  if and only if
\begin{equation}\label{E:minusculeNak}
\sum \frac{1}{\beta_1+\beta_2+\cdots +\beta_r} \cdot \ldots \cdot 
\frac{1}{\beta_1}
=\prod_{\beta \in S(w/v)}\left(1+\frac{1}{\beta}\right),
\end{equation}
\end{center}
where the sum is as in \Cref{col:NakP}.
\end{corol}
\begin{proof} First observe that by \Cref{lem:wminuscule}(2), 
since $w$ is $\pi$-minuscule, each representative $x \in W^P$ in a chain to $w$ 
is also $\pi$-minuscule. Then from \Cref{col:NakP}, we only need to show that 
any edge $x \to y$ has multiplicity $m_\pi(x,y)=1$. This follows from \Cref{cor:multweight}.
\end{proof}
If $v=id$, \Cref{col:Nak} recovers Nakada's colored hook formula stated in \Cref{thm:anotherform}. 

\Cref{col:Nak} implies 
a skew version of the classical Peterson formula \cite{carrell:vector,proctor:minuscule}. 
Recall that $\Red(w)$ 
denotes the set of all reduced expressions of $w$ and $\h(\beta)$ is the height of the positive root $\beta$.
\begin{corol}[A skew Peterson formula]\label{cor:skewPP} Let $\pi$ be a 
dominant integral weight, and let $P$ be the parabolic subgroup defined by
$\Stab_W(\pi)=W_P$. Let $v \le w\in W$ be $\pi$-minuscule elements 
such that
$Y(vW_P)$ is smooth at $wP$. Then
\begin{equation}\label{equ:petersonproctor}
\#\Red(wv^{-1})=\frac{(\ell(w)-\ell(v))!}{\prod_{\beta\in S(w/v)}\h(\beta)} ~\/.
\end{equation}
\end{corol}
\begin{proof} Consider the term of lowest degree $-\# S(w/v)$ in the expression in \Cref{col:Nak}. This corresponds to taking the summation over {\em maximal paths} in \Cref{col:Nak}, and yields
\begin{equation}\label{eq:max_path}
\sum \frac{1}{\beta_1+\beta_2+\cdots +\beta_r} \cdot \ldots \cdot 
\frac{1}{\beta_1}
=\prod_{\beta \in S(w/v)} \frac{1}{\beta} \/.
\end{equation}
The paths considered are the same as the paths from \Cref{cor:skewred}, in particular each $\beta_i$ 
is a simple root. Now specialize each simple root $\alpha_i \mapsto 1$. The right hand side gives $\frac{1}{\prod_{\beta\in S(w/v)}\h(\beta)}$. On the left hand side, each summand specializes to $\frac{1}{ (\ell(w)-\ell(v)) !}$, and the number of summands is equal to the number of reduced expressions of $wv^{-1}$, again by 
\Cref{cor:skewred}.\end{proof}

\begin{remark}\label{rem:reducetoid} Assume the hypotheses from \Cref{col:Nak}. 
It follows from \Cref{smooth:equiv_app} below 
that if $W$ is a simply laced Weyl group, then 
then $Y(v)$ is smooth at $w$ if and only if $wv^{-1}$ is $\pi'$-minuscule, where $\pi'$ is dominant integral.
Furthermore, in this case the identities \eqref{E:minusculeNak} and \eqref{equ:petersonproctor}
for $[v,w]^\pi$ coincide with those corresponding to the interval $[id,wv^{-1}]^{\pi'}$. 
If $W$ is not simply laced then it is possible that $Y(v)$ is smooth at $w$, but $wv^{-1}$ is not dominant minuscule. An 
example is in type $F_4$ for $v=s_1$, $w=s_1 s_3 s_2s_4s_3s_2s_1$, with $\alpha_1, \alpha_2$ short roots. However, in this case one can 
show that if $Y(v)$ is smooth at $w$,
the `skew' Nakada's identity is obtained from the `straight' formula applied to suitable elements 
in type $E_6$. The transformation between the two cases is given by the folding of the root system
$E_6$ into $F_4$. This suggests that folding may lead to a more general statement relating skew and straight Nakada formulae. 
\end{remark}

\begin{remark}\label{rem:general_miuscule}
If we remove the condition that $Y(vW_P)$ is smooth at $wP$, by \Cref{hook-formula}, \Cref{equ:explicit}, and the proof of \Cref{cor:skewPP}, we obtain
\begin{equation}\label{skew:cohomology}
\# \Red(wv^{-1})=(\ell(w)-\ell(v))!\frac{\sum \h(\beta_{j_1})\h(\beta_{j_2})\cdots \h(\beta_{j_{\ell(v)}})}{\h(\beta_1)\cdots \h(\beta_\ell)}
\end{equation}
where the summation is over $ 1\leq j_1<j_2<\cdots <\ell(v)\leq \ell$ such that 
$v= s_{i_{j_1}}s_{i_{j_2}}\cdots s_{i_{j_{\ell(v)}}}$ is a reduced decomposition. 
Observe also that the right hand side is the specialization of
$(\ell(w)-\ell(v))!\frac{[Y(v)]|_w}{[Y(w)]|_w}$ 
by taking a root to its height.
We will reformulate this using excited diagrams associated to a heap; see \Cref{remark:Billey_formula}.
\end{remark}

\section{Heaps, excited diagrams, and smoothness}\label{sec:heaps}

In this section, we utilize excited diagrams to give a smoothness criterion for the Schubert variety at a torus fixed point $wP$, where $w$ is a dominant minuscule element in a simply-laced Weyl group. As a consequence, we show that our skew version of the Nakada's hook formula \Cref{col:Nak} is equivalent to the original formula \Cref{thm:anotherform}, see \Cref{rem:reducetoid}.

We need to introduce the heap of a Weyl group element and excited diagrams first.
We consider a general finite Weyl group, although the statements also hold for a Kac--Moody Weyl group.
For example, the (dominant) minuscule elements (\Cref{def:minuscule})
are defined and studied in this generality
(\cite{proctor:minuscule}, \cite{stembridge2001minuscule}). The concept of $d$-complete posets
utilized by Proctor can also be included within the heap framework.

\subsection{Heap}
From now on, we fix ${\bf D}=(I,E)$ to be an unoriented Dynkin diagram, where $I$ is the node set and $E$ is the edge set of $\bf D$.
Let $W$ be the associated Weyl group of ${\bf D}$ with Cartan matrix $A=(a_{i,j})$, 
$a_{i,j}=\langle \alpha_j, \alpha_i^\vee\rangle$.
\begin{defin}(Heap $H(w)$ of a reduced expression)\label{def:heap}
Given a reduced expression 
$w=s_{i_1}s_{i_{2}}\cdots s_{i_\ell}$ of an element $w\in W$, 
define a poset $H(w)$ as follows.

As a set $H(w):=\{p_1,p_2,\ldots ,p_\ell\}$, where $\ell = \ell(w)$.
The partial order on $H(w)$ is the minimal partial order with  the property that
if  $1\leq j<k \leq \ell$ and  $s_{i_j} s_{i_k}\neq s_{i_k} s_{i_j}$ then $p_j < p_k$.
There is a (canonical) {\bf coloring} map 
\[ c_w:H(w)\to I \/; \quad c_w(p_{j})={i_j} \/; \quad  (1\leq j\leq \ell) \/.\]
The {\bf support} of $H(w)$ is the image $c_w(H(w))$. 
\end{defin}
One may think of $H(w)$ as the poset of initial subexpressions of $s_{i_1}s_{i_{2}}\cdots s_{i_\ell}$.
If $s_i = s_{i_k}$ and the index $i_k$ is clear from the context, we will say $s_i \in H(w)$ to mean that 
$c_w(p_{i_k}) = i$. 
\begin{example}\label{ex:Heap1} In the figure below,
we consider the heap $H(w)$ for a
dominant minuscule element $w\in S_{10}$
(Lie type $A_9$). The element $w$ is $\pi=\varpi_2+\varpi_7$-minuscule, and
its support has components $I_1=\{1,2,3\}, 
I_2=\{5,6,7,8,9 \}$. While $w$ is not a Grassmannian element, $w=w_1 \cdot w_2$ with
$w_1,w_2$ both Grassmannian. Here,
$w_1=s_2 s_1 s_3 s_2$ is $\varpi_2$-minuscule, and 
$w_2=s_5 s_8 s_7 s_6 s_9 s_8 s_7$ is $\varpi_7$-minuscule
(the reading order is specified by the elements $p_i$; in this case this is the reading of the colors of vertices on diagonal edges, in SE-NW direction). 
\begin{figure}[h]
\begin{tikzpicture}[x=1.2em,y=1.2em]

\foreach \x in {1,...,9}{
\node at (\x,0) {\tiny $\bullet$};
\node at (\x,-0.4) {\tiny $\x$};
\draw [dashed] (\x,0)--(\x,5);
};

\foreach \x/\y/\z in {2/1/1,1/2/2,3/2/3,2/3/4}{
\node at (\x,\y) {\tiny$\bullet$};
\node at (\x+0.1,\y+0.5) {\tiny$p_{\z}$};};
\draw [thick](1,2)--(2,3)--(3,2)--(2,1)--cycle;

\foreach \x/\y/\z in {5/2/5,8/1/6,7/2/7,6/3/8,7/4/11,8/3/10,9/2/9
}{
\node at (\x,\y) {\tiny $\bullet$};
\node at (\x+0.1,\y+0.5) {\tiny$p_{\z}$};
};

\draw [thick] (5,2)--(7,4);
\draw [thick] (6,3)--(8,1);
\draw [thick] (7,4)--(9,2);
\draw [thick] (7,2)--(8,3);
\draw [thick] (8,1)--(9,2);

\draw (1,0)--(9,0);
\end{tikzpicture}
\end{figure}
This heap may also be seen as a difference of two heaps, and it is associated to the 
skew diagram $(5,5,3,2,2)/(2,2,2)$ in $\Gr(5,10)$; see \Cref{Gr510}.
\end{example}
\begin{remark} 
Let $w$ be a $\pi$-minuscule element for dominant integral $\pi$.
It is known that each connected component (as a Hasse diagram) of $H(w)$ has a 
unique maximal element (cf. \cite[\S4]{stembridge2001minuscule}). Then
$\pi$ is the sum of fundamental weights corresponding to the maximal elements 
(see e.g. \Cref{ex:Heap1}).
\end{remark}

\begin{remark}
Originally, heaps were defined by Viennot for words of monoids and called ``Heap of pieces" \cite{viennot:heap}. Stembridge \cite{stembridge2001minuscule} considered heaps for 
non-reduced expressions of Coxeter elements. We will not need this generality.\end{remark}
From now on fix a $\pi$-minuscule element $w$ with $\pi$ integral dominant weight,
and a reduced decomposition
$w=s_{i_1} s_{i_2}\ldots s_{i_\ell}$. 
We recall next some fundamental properties of heaps $H(w)$.
For a poset $\mathcal P$, a subset $F\subset \mathcal P$ is called
an order filter if $x\in F$ and $y\in \mathcal P$ satisfying $x<y$,
then $y\in  F$. Let $\mathcal F(\mathcal P)$ be the set of order filters
of $\mathcal P$. Then $\mathcal F(\mathcal P)$ 
has a natural partial order structure by inclusion of sets.

It was proved by Stembridge \cite{stembridge2001minuscule,stembridge:fullycomm}
that the heap $H(w)$ is independent of choices of reduced expressions, 
in the following sense. Take another reduced expression $w=s_{j_1} s_{j_2}\ldots s_{j_\ell}$
with associated heap $H'(w)=\{q_1,q_2,\ldots, q_\ell\}$ and
labeling $c'_w:H'(w)\to I$ as in \Cref{def:heap}.
Then there exists a permutation $\sigma$ of $\{1,2,\ldots, \ell\}$ such that
the induced bijective map ${\sigma}:H(w)\to H'(w), p_i \mapsto q_{\sigma(i)}$
gives an isomorphism of posets with compatible coloring maps, 
i.e. $c'_w\circ {\sigma}=c_w$. 

Consider $v\in [id, w]^\pi$. Since by \Cref{prop:weakstrong} 
intervals in weak and strong order coincide, $v$ has a canonical 
reduced expression obtained from a final subexpression of 
$s_{i_1} s_{i_2}\ldots s_{i_\ell}$.
Then $H(v)$ is canonically identified
with an order filter of $H(w)$, and 
$c_v$ is the restriction of $c_w$. Further, $H(wv^{-1}) = H(w) \setminus H(v)$ 
and we have the following poset isomorphism:
\begin{equation}\label{E:FHiso}
[v, w]^\pi \simeq \mathcal F(H(w)\setminus H(v))\/, \quad z\mapsto H(z)\setminus H(v) \/.
\end{equation}
\subsection{Excited diagrams of heaps} Excited (Young) diagrams were defined in \cite{ikeda2009excited} as
a tool to calculate equivariant localization of Schubert classes in the maximal orthogonal Grassmannians 
of classical types. Recently, their definition was extended in \cite[Def.3.1]{naruse2019skew} to include 
excited diagrams of any colored poset. We recall this definition next, and prove 
some properties relevant to this paper.

\begin{defin}\label{def:excite_dg}
Let $(\mathcal P,c)$ be a poset equipped with a coloring map $c:\mathcal{P} \to I$
with values in the vertices of the Dynkin diagram ${\bf D}=(I,E)$. Define the partial order 
$\unrhd$
on the subsets of $\mathcal{P}$ generated by the following covering relation denoted by 
$\vartriangleright$
(and called {\bf elementary excitation}). 

If $D, D'\subset \mathcal P$ then 
$D \vartriangleright D'$ (and we say $D$ is excited to $D'$) if the following hold: 
\begin{enumerate}
\item There exist $p<q$ with $p \in D'$ and $q\in D$ such that $D\setminus \{q\}=D'\setminus \{p\}$, 
the colors of $p$ and $q$ are the same, and furthermore, there is no other element  
$p<z<q$ in $\mathcal{P}$ such that $c(z) = c(p)$.
 \item For any $u \in D \cap [p,q]$, the pair $(c(p), c(u))$ is not an edge in the Dynkin diagram ${\bf D}$. 
\end{enumerate} 
For $F\subset \mathcal P$, the set $\mathcal E_{\mathcal P} (F)$ 
of excited diagrams of $F$ in $\mathcal P$ is the principal order ideal generated by $F$:
$\mathcal E_{\mathcal P} (F) = \{ D: F \unrhd D\}$.
\end{defin}

An element $D$ in $\mathcal E_{\mathcal P} (F)$ is called an {\bf excited diagram} of $F$ in $\mathcal P$.
Our main examples of excited diagrams come from heaps of minuscule elements $(H(w), c_w)$. 

\begin{example}\label{ex:typeA_d-comp} Consider 
$w=s_2 s_1 s_3 s_2s_5 s_8 s_7 s_6 s_9 s_8 s_7$ and 
$H(w)$ from \Cref{ex:Heap1}. Take $v=s_2 s_8s_7 \le w$. 
The heap $H(v)$ is represented by the red dots 
of $H(w)$ in the first diagram below. The collection of all diagrams forms the set
$\mathcal E_{H(w)} (H(v))$ of excited diagrams of $H(v)$. 
%%%%%%%%%%%%%%%%%%%%%%%%%%%%%%%%%%%%%%
\begin{figure}[h]
\begin{tikzpicture}[x=1.0em,y=1.0em]%%% D1 %%%%%%%%%%
\foreach \x in {1,...,9}{
\node at (\x,0) {\tiny $\bullet$};
\node at (\x,-0.4) {\tiny $\x$};
\draw [dashed] (\x,0)--(\x,5);
};
\foreach \x/\y/\z in {2/1/1,1/2/2,3/2/3,2/3/4}{
\node at (\x,\y) {\tiny$\bullet$};
\node at (\x+0.1,\y+0.5) {\tiny$p_{\z}$};
};
\draw [thick](1,2)--(2,3)--(3,2)--(2,1)--cycle;

\foreach \x/\y/\z in {5/2/5,8/1/6,7/2/7,6/3/8,7/4/11,8/3/10,9/2/9
}{
\node at (\x,\y) {\tiny $\bullet$};
\node at (\x+0.1,\y+0.5) {\tiny$p_{\z}$};
};
\draw [thick] (5,2)--(7,4);
\draw [thick] (6,3)--(8,1);
\draw [thick] (7,4)--(9,2);
\draw [thick] (7,2)--(8,3);
\draw [thick] (8,1)--(9,2);
\draw (1,0)--(9,0);

\foreach \x/\y in {2/3,7/4,8/3}{
\node[red] at (\x,\y) {\Large$\bullet$};};

\node[red] at (4,4.2) {\Small $D_1=H(v)$};

\end{tikzpicture}
\hspace{1cm}
\begin{tikzpicture}[x=1.0em,y=1.0em]%%% D2 %%%%%%%%%%
\foreach \x in {1,...,9}{
\node at (\x,0) {\tiny $\bullet$};
\node at (\x,-0.4) {\tiny $\x$};
\draw [dashed] (\x,0)--(\x,5);
};
\foreach \x/\y/\z in {2/1/1,1/2/2,3/2/3,2/3/4}{
\node at (\x,\y) {\tiny$\bullet$};
\node at (\x+0.1,\y+0.5) {\tiny$p_{\z}$};
};
\draw [thick](1,2)--(2,3)--(3,2)--(2,1)--cycle;

\foreach \x/\y/\z in {5/2/5,8/1/6,7/2/7,6/3/8,7/4/11,8/3/10,9/2/9
}{
\node at (\x,\y) {\tiny $\bullet$};
\node at (\x+0.1,\y+0.5) {\tiny$p_{\z}$};
};
\draw [thick] (5,2)--(7,4);
\draw [thick] (6,3)--(8,1);
\draw [thick] (7,4)--(9,2);
\draw [thick] (7,2)--(8,3);
\draw [thick] (8,1)--(9,2);
\draw (1,0)--(9,0);

\foreach \x/\y in {2/3,7/4,8/1}{
\node[red] at (\x,\y) {\Large$\bullet$};};

\node[red] at (4,4.2) {\Small$D_2$};

\end{tikzpicture}
\hspace{1cm}%
\begin{tikzpicture}[x=1.0em,y=1.0em]%%% D3 %%%%%%%%%%
\foreach \x in {1,...,9}{
\node at (\x,0) {\tiny $\bullet$};
\node at (\x,-0.4) {\tiny $\x$};
\draw [dashed] (\x,0)--(\x,5);
};
\foreach \x/\y/\z in {2/1/1,1/2/2,3/2/3,2/3/4}{
\node at (\x,\y) {\tiny$\bullet$};
\node at (\x+0.1,\y+0.5) {\tiny$p_{\z}$};
};
\draw [thick](1,2)--(2,3)--(3,2)--(2,1)--cycle;

\foreach \x/\y/\z in {5/2/5,8/1/6,7/2/7,6/3/8,7/4/11,8/3/10,9/2/9
}{
\node at (\x,\y) {\tiny $\bullet$};
\node at (\x+0.1,\y+0.5) {\tiny$p_{\z}$};
};
\draw [thick] (5,2)--(7,4);
\draw [thick] (6,3)--(8,1);
\draw [thick] (7,4)--(9,2);
\draw [thick] (7,2)--(8,3);
\draw [thick] (8,1)--(9,2);
\draw (1,0)--(9,0);

\foreach \x/\y in {2/3,7/2,8/1}{
\node[red] at (\x,\y) {\Large$\bullet$};};
\node[red] at (4,4.2) {\Small$D_3$};

\end{tikzpicture}
\vspace{0.3cm}

\begin{tikzpicture}[x=1.0em,y=1.0em]%%% D4 %%%%%%%%%%
\foreach \x in {1,...,9}{
\node at (\x,0) {\tiny $\bullet$};
\node at (\x,-0.4) {\tiny $\x$};
\draw [dashed] (\x,0)--(\x,5);
};
\foreach \x/\y/\z in {2/1/1,1/2/2,3/2/3,2/3/4}{
\node at (\x,\y) {\tiny$\bullet$};
\node at (\x+0.1,\y+0.5) {\tiny$p_{\z}$};
};
\draw [thick](1,2)--(2,3)--(3,2)--(2,1)--cycle;

\foreach \x/\y/\z in {5/2/5,8/1/6,7/2/7,6/3/8,7/4/11,8/3/10,9/2/9
}{
\node at (\x,\y) {\tiny $\bullet$};
\node at (\x+0.1,\y+0.5) {\tiny$p_{\z}$};
};
\draw [thick] (5,2)--(7,4);
\draw [thick] (6,3)--(8,1);
\draw [thick] (7,4)--(9,2);
\draw [thick] (7,2)--(8,3);
\draw [thick] (8,1)--(9,2);
\draw (1,0)--(9,0);

\foreach \x/\y in {2/1,7/4,8/3}{
\node[red] at (\x,\y) {\Large$\bullet$};};

\node[red] at (4,4.2) {\Small$D_4$};

\end{tikzpicture}
\hspace{1cm}
\begin{tikzpicture}[x=1.0em,y=1.0em]%%% D5 %%%%%%%%%%
\foreach \x in {1,...,9}{
\node at (\x,0) {\tiny $\bullet$};
\node at (\x,-0.4) {\tiny $\x$};
\draw [dashed] (\x,0)--(\x,5);
};
\foreach \x/\y/\z in {2/1/1,1/2/2,3/2/3,2/3/4}{
\node at (\x,\y) {\tiny$\bullet$};
\node at (\x+0.1,\y+0.5) {\tiny$p_{\z}$};
};
\draw [thick](1,2)--(2,3)--(3,2)--(2,1)--cycle;

\foreach \x/\y/\z in {5/2/5,8/1/6,7/2/7,6/3/8,7/4/11,8/3/10,9/2/9
}{
\node at (\x,\y) {\tiny $\bullet$};
\node at (\x+0.1,\y+0.5) {\tiny$p_{\z}$};
};
\draw [thick] (5,2)--(7,4);
\draw [thick] (6,3)--(8,1);
\draw [thick] (7,4)--(9,2);
\draw [thick] (7,2)--(8,3);
\draw [thick] (8,1)--(9,2);
\draw (1,0)--(9,0);

\foreach \x/\y in {2/1,7/4,8/1}{
\node[red] at (\x,\y) {\Large$\bullet$};};

\node[red] at (4,4.2) {\Small$D_5$};

\end{tikzpicture}
\hspace{1cm}%
\begin{tikzpicture}[x=1.0em,y=1.0em]%%% D6 %%%%%%%%%%
\foreach \x in {1,...,9}{
\node at (\x,0) {\tiny $\bullet$};
\node at (\x,-0.4) {\tiny $\x$};
\draw [dashed] (\x,0)--(\x,5);
};
\foreach \x/\y/\z in {2/1/1,1/2/2,3/2/3,2/3/4}{
\node at (\x,\y) {\tiny$\bullet$};
\node at (\x+0.1,\y+0.5) {\tiny$p_{\z}$};
};
\draw [thick](1,2)--(2,3)--(3,2)--(2,1)--cycle;

\foreach \x/\y/\z in {5/2/5,8/1/6,7/2/7,6/3/8,7/4/11,8/3/10,9/2/9
}{
\node at (\x,\y) {\tiny $\bullet$};
\node at (\x+0.1,\y+0.5) {\tiny$p_{\z}$};
};
\draw [thick] (5,2)--(7,4);
\draw [thick] (6,3)--(8,1);
\draw [thick] (7,4)--(9,2);
\draw [thick] (7,2)--(8,3);
\draw [thick] (8,1)--(9,2);
\draw (1,0)--(9,0);

\foreach \x/\y in {2/1,7/2,8/1}{
\node[red] at (\x,\y) {\Large$\bullet$};};

\node[red] at (4,4.2) {\Small $D_6$};

\end{tikzpicture}

$\mathcal E_{H(w)}(H(v))=\left\{
\begin{minipage}{2.7cm}
\begin{tikzpicture}\small
\node[red] at (1,2) {$D_1\vartriangleright D_2 \vartriangleright D_3$};
\node[red] at (1,1) {$D_4\vartriangleright D_5 \vartriangleright D_6$};
\foreach \x/\y in {0/1.5,1/1.5,1.9/1.5}{
\node[red] at (\x,\y) {\tiny$\bigtriangledown$};};
\end{tikzpicture}
\end{minipage}
\right\}
$

\caption{
$\mathcal E_{H(w)}(H(v))$, 
$w=s_2 s_1 s_3 s_2 \cdot s_5 s_8 s_7 s_6 s_9 s_8 s_7$, {$v=s_2 \cdot s_8 s_7$}.
}\label{ex:a9excite}
\end{figure}
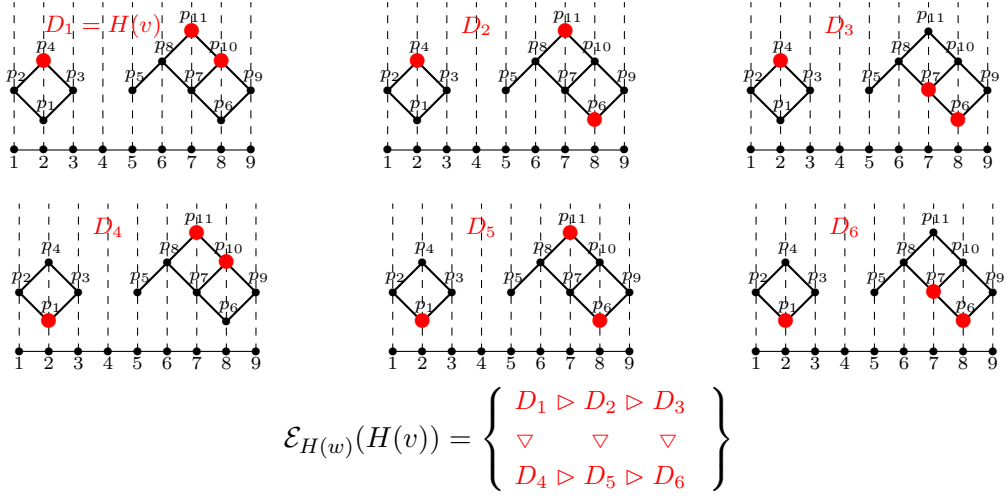 
\end{example}
We recall next a characterization of dominant minuscule elements due to Stembridge \cite{stembridge2001minuscule} (cf.
Prop.~2.3 and Prop.~2.5 in {\em loc.cit.}).

\begin{prop}\label{prop:stembridge-crit} Let $w = s_{i_1} \cdot \ldots \cdot s_{i_p}$ be
a reduced expression. Then the following hold:

(a) The element $w$ is minuscule if and only if between any pair of occurrences
of a generator $s_i$, with no other $s_i$ in between, there are exactly:
\begin{itemize} \item two terms $s_{j_1},s_{j_2}$ that do not commute with $s_i$, and the 
corresponding simple roots $\alpha_{j_1},\alpha_{j_2}$ have lengths at most that of $\alpha_i$, or,
\item one term that does not commute with $s_i$, say $s_j$, and $\langle \alpha_j, \alpha_i^\vee\rangle =-2$.
\end{itemize}

(b) The element $w$ is dominant minuscule if the conditions in (a) are satisfied, and, in addition,
the last occurrence of each generator $s_i$ is followed by at most one term $s_j$ 
that does not commute with $s_i$, and the corresponding root $\alpha_j$ has length at most that 
of $\alpha_i$. 
\end{prop}

\begin{theorem}\label{cor:EYD_minuscule}
Let $w$ be a $\pi$-minuscule element for some dominant integral weight $\pi$, and
let $v\in [id, w]^\pi$.
Then the following are equivalent:

(a) There are no nontrivial excited diagrams of $H(v)$ in $H(w)$, i.e., $|\mathcal E_{H(w)}(H(v))|=1$; 

(b) $wv^{-1}$ is dominant minuscule for some dominant integral weight $\pi'$.
\end{theorem}
\begin{proof} By \Cref{prop:weakstrong} we may choose a reduced expression
$w=s_{i_1}\cdots s_{i_\ell}$ of $w$ such that the expressions
$v=s_{i_{j+1}}\cdots s_{i_\ell}$ and 
$w v^{-1}=s_{i_1}\cdots s_{i_j}$ are reduced. By hypothesis on $w$,
$w v^{-1}$ is  $v(\pi)$-minuscule. Assume that $|\mathcal E_{H(w)}(H(v))|=1$.
To show $w v^{-1}$ is dominant minuscule, we check the conditions in 
\Cref{prop:stembridge-crit}(b) to 
the heap $H(wv^{-1}) = H(w) \setminus H(v)$.
Let $s_i$ be a last occurrence of a simple reflection in $wv^{-1}$.
If $s_i$ is also a last occurrence in $w$ we are done, since $w$ is dominant 
minuscule. Therefore we may assume that $s_i$ occurs again in $v$, and we
pick the first such occurrence. Since $w$ is minuscule, by \Cref{prop:stembridge-crit}(a),
between these two consecutive occurrences of $s_i$, 
there is exactly one element $s_{j_1}$ or two elements $s_{j_1},s_{j_2}$ which do not 
commute with $s_i$. But at least one of $s_{j_1},s_{j_2}$ must be in $v$; otherwise one can construct
an excited diagram by moving the element corresponding to 
$s_i$ from $H(v)$ into the one from $H(wv^{-1})$.
The remaining element (say) $s_{j_2}$ in $wv^{-1}$ which does not commute with $s_i$ (if it exists) 
must necessarily arise from the first condition in \Cref{prop:stembridge-crit}(a), which implies that the length
of the corresponding simple root $\alpha_{j_2}$ is at most that of $\alpha_i$. This verifies the conditions in 
\Cref{prop:stembridge-crit}(b).

We now prove the opposite implication. From \Cref{def:excite_dg} of excited diagrams, 
if $|\mathcal E_{H(w)}(H(v))|>1$ there must be two consecutive occurrences of a reflection 
$s_i$ in $w$, such that one is in $v$, and the other in $wv^{-1}$. Furthermore,
there cannot be any reflection $s_j \in H(v)$ between the two $s_i$'s such that $i,j$ are adjacent 
in the Dynkin diagram. At the same time, since $w$ is minuscule, 
there must be exactly one or two such $s_j$'s between the $s_i$'s in $w$, and if there is 
only one, then $\langle \alpha_j , \alpha_i^\vee \rangle = -2$. Then these $s_j$'s 
must be in $wv^{-1}$. Both situations contradict the condition in 
\Cref{prop:stembridge-crit}(b), thus $wv^{-1}$ is not dominant minuscule.  
 \end{proof}
\begin{lemma}\label{EYD:subexpr}\cite[Proposition 3.4, Remark 3.18]{naruse2019skew}
Let $w=s_{i_1}s_{i_{2}}\cdots s_{i_\ell}$ be a reduced expression for a $\pi$-minuscule element, 
with $\pi$ a dominant integral weight. For any $v\in W^\pi$ satisfying $v\leq w$, there is a bijective correspondence between the set of excited diagrams in
$\mathcal E_{H(w)}(H(v))$ and subexpressions of $s_{i_1}s_{i_{2}}\cdots s_{i_\ell}$ 
which give a reduced expression of $v$.
\end{lemma}
If $w \in W$ and $H(w) = \{ p_1, \ldots, p_{\ell(w)}\}$, define $\beta_w : H(w)\to S(w)$, by
$\beta_w(p_j)=s_{i_1} s_{i_2}\cdots s_{i_{j-1}} \alpha_{i_j}$, for $1\leq j\leq \ell$.
\begin{remark}\label{remark:Billey_formula} Under the assumptions in the \Cref{EYD:subexpr},
Billey's formula \cite{billey:kostant} can be written as follows:
\begin{equation}\label{equ:billey}
[Y(v)]|_w= \sum_{D\in \mathcal E_{H(w)}(H(v))} \prod_{p \in D} 
\beta_w (p).
\end{equation}
Using this we may reformulate \Cref{skew:cohomology} as follows.
\begin{equation}\label{skew:NO}
\# \Red(wv^{-1}) = (\ell(w)-\ell(v))! \sum_{D\in \mathcal E_{H(w)}(H(v))} \frac{1}{\prod_{p \in H(w)\setminus D} 
ht(\beta_w(p))}.
\end{equation}
\end{remark}

\subsection{Main results}\label{subsec:main}
Finally, we can state the main results in this section. Recall that for any point $p$ in a variety $X$, there is a notion of multiplicity $\mult_p (X)$ with the property that $p$ is smooth if and only if $\mult_p(X)=1$; we 
refer e.g. to \cite[Section 4.7]{billey.lakshimibai:singula_loci} for details.

Throughout this section we fix $w$ a $\pi$-minuscule element where $\pi$ is dominant integral.
As usual, let $P$ be the parabolic group defined by $\Stab_{W}(\pi)=W_P$.
\begin{theorem}\label{smooth:equiv_app}
Assume that $W$ is a simply-laced Weyl group and let $v \in W^P$ such that $v \leq w$.
Then the following conditions are equivalent:

(1)  $Y(v) \subset G/P$ is non-singular at $wP$;

(2) There are no nontrivial excited diagrams of $H(v)$ in $H(w)$, i.e., $|\mathcal E_{H(w)}(H(v))|=1$;

(3) $w v^{-1}$ is $\pi'$-minuscule for some dominant integral weight $\pi'$.

\noindent
Moreover these conditions imply that for any $z \in [v,w]^P$, {$S(z/v)=S(z v^{-1})$}. 

Furthermore, in this case,
the $\Lambda$-Bruhat graphs for $[v,w]^P$ and $[id, w v^{-1}]^{P'}$, 
equipped with the {constant admissible functions 
$\Lambda\equiv \pi$ respectively $\Lambda'\equiv \pi'$},
%standard admissible functions from \Cref{E:standardLa}, 
%{\color{blue} $\Lambda\equiv \pi$ respectively $\Lambda'\equiv \pi'$}
are isomorphic. Here
$P'$ is the parabolic subgroup defined by $\Stab_{W}(\pi')=W_{P'}$. 
\end{theorem}

For the proof, we need some results of Graham and Kreiman \cite{graham.kreiman:cominuscule_pt}.
An element $x \in W$ is {\bf cominuscule} if there exists $h \in \mathfrak{h}$ such that for any root
$\beta$ in the inversion set of $x^{-1}$, we have $\beta(h) = -1$ 
(see~\cite[Definition 5.6]{graham.kreiman:cominuscule_pt}). 
As observed in {\em loc.cit.} this is equivalent 
to the notion of a (Peterson) minuscule element, and the two notions coincide for 
simply laced Weyl groups.
Indeed, in this case we may identify $\mathfrak h^{*} \simeq \mathfrak h$ by 
$\beta \leftrightarrow \beta^\vee$, and we take  
$h_o=w(\pi)\in \mathfrak h$. Then from \Cref{lemma:ubetav},
\begin{equation}\label{E:comin}
\beta(h_o)= \langle \beta, h_o \rangle =\langle w(\pi),\beta^\vee \rangle=-1 \/.
\end{equation}

This motivates the definition of 
$\varphi_{w,\pi}:S(\mathfrak h)\to \mathbb C$ given by
$\varphi_{w,\pi}(h):=\langle -w(\pi),h \rangle$.
Then for any $\beta \in S(w)$, $\varphi_{w,\pi}(\beta)=1$;
\begin{lemma}\label{lemma:eqmult} 
Assume the hypotheses of \Cref{smooth:equiv_app}. Then 
the multiplicity of $Y(v) \subset G/P$ at $w$ equals to:
\begin{equation*}\label{m_v(w)}
\mult_{w}(Y(v))=\varphi_{w,\pi}([Y(v)]|_{w})=|\mathcal E_{H(w)}(H(v))| \/.
\end{equation*}
\end{lemma}

\begin{proof} Denote by $Y(vB)$ the corresponding 
Schubert variety in $G/B$. Since the projection map $G/B \to G/P$ is a 
smooth map, 
$\mult_{wP}(Y(vP))=\mult_{wB}(Y(vB))$. Furthermore,
$[Y(v)]|_w = [Y(wB)]|_{wB}$ as elements in $H^*_T(pt)$. 
Then the first equality follows from \cite[Corollary 5.12]{graham.kreiman:cominuscule_pt},
and the second from \Cref{equ:billey}.
\end{proof}
\begin{example} In \Cref{ex:typeA_d-comp}, there are $6$ excited diagrams,
showing that the multiplicity of $Y(v)$ at $w$ equals to $6$.
\end{example}

\begin{proof}[Proof of \Cref{smooth:equiv_app}] 
The equivalence $(1) \Longleftrightarrow (2)$
follows from \Cref{lemma:eqmult}, 
and the equivalence $(2) \Longleftrightarrow (3)$ from
\Cref{cor:EYD_minuscule} 
(for any Weyl group $W$).

To prove the remaining claims, first observe that 
since $Y(v)$ is non-singular at $w$, it will be non-singular
at any $z \in [v,w]^P$. To prove that $S(z/v)= S(zv^{-1})$, 
we utilize the equation \eqref{eq:max_path} in two
situations: the intervals $[v,z]^P$ and $[id,zv^{-1}]^{P'}$. 
By \Cref{cor:skewred}, there is a weight preserving 
bijection between the maximal length paths in $[v,z]^P$ 
and $[id,z v^{-1}]^{P'}$. Consider the equality in \Cref{col:Nak},
applied to the intervals above, and restricted to the terms of smallest degree; 
this corresponds to taking the sum over maximal length paths 
on the left hand side. The corollary and the considerations above imply that
\[ \prod_{\beta \in S(z v^{-1})} \beta = \prod_{\beta' \in S(z/v)} \beta' \/. \]
Then the claim follows because all roots $\beta, \beta'$ are positive and in each side they are all mutually prime when regarded 
as linear forms.

Finally, we prove the isomorphism between the standard
$\Lambda$-Bruhat graphs on $[v, w]^P$ and $[id, w v^{-1}]^{P'}$. 
First of all, we have a poset isomorphism $\Phi_{v,w}:[v, w]^P\to [id, w v^{-1}]^{P'}$ defined by
$z \mapsto z v^{-1}$ and induced by the isomorphism in \eqref{E:FHiso}:
\[[v, w]^\pi \simeq \mathcal F(H(w)\setminus H(v))=\calF(H(wv^{-1}))\simeq [id,wv^{-1}]^{\pi'} \/.\]
We prove next that the edges in the two $\Lambda$-Bruhat graphs match. 
Suppose $x \overset{\beta}{\to} y$ is an edge in the $\Lambda$-Bruhat graph of $[v,w]^P$.
Since $Y(v)$ is non-singular at $wP$, {$Y(v)$ will be non-singular at $yP$}, thus 
$\beta\in S(y/v)=S(yv^{-1})$. Thus, there is an edge with 
target $yv^{-1}$ and weight $\beta$ in the $\Lambda$-Bruhat graph of $[id,wv^{-1}]^{P'}$. 
Suppose the source of this edge is 
$zv^{-1}$ for some $z\in [v,w]^P$. By \Cref{cor:multweight}(2), 
\[\calW_\pi(x)-\calW_\pi(y)=\beta=\calW_{\pi'}(zv^{-1})-\calW_{\pi'}(yv^{-1}) \/, \]
and the last expression equals $\sum \alpha_i$, where
$\alpha_i$'s appear in any maximal chain from $zv^{-1}$ to $yv^{-1}$. By \Cref{cor:skewred},
such chains are in bijection to maximal chains from $z$ to $y$, and contain the same weights. 
This proves that  
\[  \calW_{\pi'}(zv^{-1})-\calW_{\pi'}(yv^{-1})  = \calW_\pi(z)-\calW_\pi(y) \/. \]
Hence, $\calW_\pi(x)=\calW_\pi(z)$, thus $x=z$ by the injectivity of $\calW_\pi$, see \Cref{cor:multweight}(3).
Same argument, using maximal chains this time in the intervals $[y,w]^\pi$ and $[yv^{-1},wv^{-1}]^{\pi'}$, 
and taking into account that $\calW_{\pi'}(wv^{-1})= \calW_\pi(w)=0$,
shows that
$\calW_{\pi'}(yv^{-1})= \calW_\pi(y)$, 
and it finishes the proof. \end{proof}

\section{Examples}\label{sec:examples} In this section we give some examples illustrating \Cref{thm:genNak} and its corollaries, and the heap constructions from \S \ref{sec:heaps}. 

\subsection{Grassmannians} Let $X= G/P=\Gr(k,n)$ be a Grassmann manifold, where
$P$ is the maximal parabolic subgroup with $\Delta_P = \Delta \setminus \{ \alpha_k \}$. 
The permutations in $W^P$ are in bijection to the set of partitions 
$\mu = (\mu_1, \ldots, \mu_k)$ in the $k \times (n-k)$ rectangle.
We denote by $w_\mu$ the element in $W^P$ corresponding to $\mu$. This bijection has the property that
$\ell(w_\mu) = |\mu| = \mu_1 + \ldots + \mu_k$ (the number of boxes of $\mu$) and $w_\mu < w_\nu$ if and 
only if $\mu \subset \nu$. Furthermore, there exists an edge $w_\mu \to w_\nu$ in the Bruhat graph if and only if the skew shape
$\nu/\mu$ is a {\em rim hook} in $\nu$. Recall that a rim hook in $\nu$ is a connected collection $B$ of 
boxes which intersect the boundary of $\nu$ in at least one point and such that $\nu \setminus B$ remains a partition. 
One may read the reduced decompositions of 
$w_\mu,w_\nu$ (cf.~\cite{ikeda2009excited,BCMP:QKChev}, the latter, for more general cominuscule Grassmannians) and the root $\beta$ such that $w_\nu W_P = s_\beta w_\mu W_P$ directly from the diagrams involved. We illustrate this in the example below.
\begin{example}\label{example:Gr512} Take $G/P = \Gr(5,12)$; then $\Delta_P = \{ \alpha_i: 1 \le i \le 11\} \setminus \{ \alpha_5 \}$. Consider $\nu = (7,7,7,5,5)$ included in the $5 \times 7$ rectangle. The partition $\nu$ corresponds to the green portion in the left diagram below.
\begin{center}
\begin{tabular}{cc}
\ytableausetup{centertableaux,boxsize=normal}
\begin{ytableau}
*(green)  \alpha_5   & *(green){\alpha_6}  & *(green){\alpha_7}
& *(green){ \alpha_8} & *(green) {\alpha_9} & *(green){\alpha_{10}} & *(green) {\alpha_{11}}\\
*(green) {\alpha_4} & *(green) { \alpha_5} & *(green){\alpha_6} & *(green){ \alpha_7} &
*(green) {\alpha_8} & *(green) {\alpha_{9}} & *(green) {\alpha_{10}}\\
*(green) {\alpha_3} & *(green) { \alpha_4} & *(green){\alpha_5} & *(green) { \alpha_6} &
*(green){\alpha_7} & *(green){\alpha_{8}} & *(green) {\alpha_{9}}\\
*(green) {\alpha_2} & *(green) { \alpha_3} & *(green){\alpha_4} & *(green){ \alpha_5} &
*(green) {\alpha_6} & {\alpha_{7}} & {\alpha_{8}}\\
*(green) {\alpha_1} & *(green) { \alpha_2} & *(green){\alpha_3} & *(green){ \alpha_4} &
*(green) {\alpha_5} & {\alpha_{6}} & {\alpha_{7}}
\end{ytableau} \hskip2cm
\ytableausetup{centertableaux,boxsize=normal}
\begin{ytableau}
\alpha_5   & {\alpha_6}  & {\alpha_7}
& { \alpha_8} & {\alpha_9} & *(yellow){\alpha_{10}} & *(green) {\alpha_{11}}\\
{\alpha_4} & { \alpha_5} & {\alpha_6} & { \alpha_7} & *(yellow)
{\alpha_8} & *(yellow) {\alpha_{9}} & *(green) {\alpha_{10}}\\
 {\alpha_3} & *(orange) { \alpha_4} & *(orange){\alpha_5} & *(cyan) { \alpha_6} &
*(green){\alpha_7} & *(green){\alpha_{8}} & *(green) {\alpha_{9}}\\
*(orange) {\alpha_2} & *(orange) { \alpha_3} & *(cyan){\alpha_4} & *(cyan){ \alpha_5} &
*(green) {\alpha_6} & {\alpha_{7}} & {\alpha_{8}}\\
*(green) {\alpha_1} & *(green) { \alpha_2} & *(green){\alpha_3} & *(green){ \alpha_4} &
*(green) {\alpha_5} & {\alpha_{6}} & {\alpha_{7}}
\end{ytableau}
\end{tabular}
\end{center}
One may read a reduced decomposition for $w_\nu$ by reading the labels of $\nu$ bottom to top, right to left:
$w_\nu = s_5s_4s_3s_2s_1s_6s_5s_4s_3s_2s_9s_8s_7s_6s_5s_4s_3s_{10}s_9s_8s_7s_6s_5s_4s_{11}s_{10}s_9s_8s_7s_6s_5$. Consider $\mu = (5,4,1)$. A path in the Bruhat graph from $\mu$ to $\nu$ is given by removing  rim hooks starting from $\nu$. One example is the path below, where the rim hooks removed are, in order, starting from $\nu$:
{\color{green} green}, {\color{cyan} cyan}, {\color{orange} orange}, {\color{yellow} yellow}.
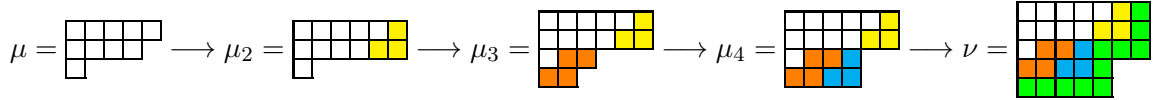
\[
\ytableausetup{centertableaux,boxsize=0.6em}
\mu=\begin{ytableau} 
{}& {} & {} & {} & {} \\
{}& {} & {} & {} \\ {} 
\end{ytableau}
\longrightarrow  
\mu_2=\begin{ytableau}
{} & {} & {} & {} & {} & *(yellow) {} \\
{} & {} & {} & {} & *(yellow) {} & *(yellow) {}\\ {} 
\end{ytableau} 
\longrightarrow  
\mu_3=\begin{ytableau}
 {} & {}  & {}
& {} & {} & *(yellow){} \\
{} & {} & {} & {} & *(yellow)
{} & *(yellow) {} \\
 {} & *(orange) {} & *(orange){} \\
*(orange) {} & *(orange) {}
\end{ytableau}
\longrightarrow
\mu_4=\begin{ytableau}
{}  & {}  & {}
& {} & {} & *(yellow){} \\
{} & {} & {} & {} & *(yellow)
{} & *(yellow) {} \\
 {} & *(orange) {} & *(orange){} & *(cyan) {}\\
*(orange) {} & *(orange) {} & *(cyan){} & *(cyan){}
\end{ytableau}
\longrightarrow
\nu=\begin{ytableau}
{}  & {}  & {}
& {} & {} & *(yellow){} & *(green) {}\\
{} & {} & {} & {} & *(yellow)
{} & *(yellow) {} & *(green) {}\\
 {} & *(orange) {} & *(orange){} & *(cyan) {} &
*(green){} & *(green){} & *(green) {}\\
*(orange) {} & *(orange) {} & *(cyan){} & *(cyan){} &
*(green) {} \\
*(green) {} & *(green) {} & *(green){} & *(green){} &
*(green) {} 
\end{ytableau}
\ytableausetup{centertableaux,boxsize=normal}
\]
The root $\beta$ associated to an edge $\lambda_1 \to \lambda_2$ is the sum of the simple roots in the 
boxes of the skew shape $\lambda_2/\lambda_1$. In our case, if $w_\nu W_P= s_\beta w_{\mu_4} W_P$
then $\beta = \sum_{i=1}^{11} \alpha_i$ (the sum of the labels of the green rim hook).   
\end{example}  
In order to illustrate \Cref{col:Nak} we consider the standard admissible function $\Lambda \equiv \omega_P = \omega_k$. Then each element $w \in W^P$ is $\omega_k$-minuscule, and all edge multiplicities are 
equal to $1$. Further, $Y(w_\mu W_P)$ is smooth at $w_\nu W_P$ if and only if the skew partition $\nu/\mu$ is 
a union of straight shapes. (This may be deduced e.g. from \cite[Corollary 9.2, 9.3]{ikeda2009excited} or from \cite{graham.kreiman:excited}.) In the example above, $Y(w_\mu)$ is singular at $w_\nu$. 
\begin{example}\label{ex:Gr(2,5)}
We now consider $G/P =\Gr(2,5)$ and the standard admissible function $\Lambda\equiv \omega_2$. 
Take $ v:=w_\emptyset <  w:=w_{\tableau{5}{ {} & {}& {} \\ {} & {} & {}}}$; of course $Y(v)=G/P$ is smooth.  
Consider the path
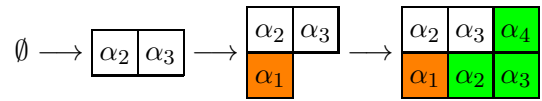
\[ \emptyset \longrightarrow \begin{ytableau} \alpha_2 & \alpha_3 \end{ytableau} \longrightarrow  \begin{ytableau} \alpha_2 & \alpha_3 \\ *(orange) \alpha_1 
\end{ytableau} \longrightarrow  \begin{ytableau}
\alpha_2   & {\alpha_3}  & *(green){\alpha_4} \\
*(orange) {\alpha_1} & *(green) { \alpha_2} & *(green) {\alpha_3} \end{ytableau}
\]
As before, we indicated the removed hooks with colors, and also their box labels.
In the summation from \Cref{col:Nak} this corresponds to the product
\[ \frac{1}{\beta_1 + \beta_2 + \beta_3} \times\frac{1}{\beta_1 + \beta_2} \times \frac{1}{\beta_1} = \frac{1}{\alpha_1+2 \alpha_2 + 2 \alpha_3 +\alpha_4}\times \frac{1}{\alpha_1+\alpha_2 + \alpha_3 +\alpha_4} \times \frac{1}{\alpha_2+\alpha_3+\alpha_4} \/. \]
\end{example}
 
\subsection{Submaximal isotropic Grassmannians} In this section we consider $G/P= \mathrm{IG}(k,2n)$,
the space parametrizing $k$-dimensional subspaces of $\C^{2n+1}$ isotropic with 
respect to a symmetric non-degenerate bilinear form on $\C^{2n+1}$.~Then 
$\Delta = \{ \alpha_i: 1 \le i \le n \}$ and by convention $\alpha_n$ is short. 
In this case, $\Delta_P = \Delta \setminus \{\alpha_k\}$ and the multiplicities 
may be $2$. This is an example of a `non-minuscule' flag manifold, and the equality from 
\Cref{col:NakP} is in general stronger than the Nakada's colored hook formula. 

We consider the standard admissible function $\Lambda \equiv \omega_k$.
The root $\beta$ for an edge $u \to u'$ such that $u'W_P= s_\beta u W_P$ may be found 
from \Cref{E:Wla}:
\begin{equation*}\label{E:betaIG} 
\beta = \frac{\mathcal{W}_\Lambda(u) - \mathcal{W}_\Lambda(u')}{m(u,u')} = \frac{u (\omega_k) - u'(\omega_k)}{m(u, u')} \/. 
\end{equation*}

To illustrate, consider $\mathrm{IG}(2,7)$ and the standard admissible function $\Lambda \equiv \omega_2$.
Recall the $\Lambda$-Bruhat graph from \Cref{ex:B3} above, where the simple edges have multiplicity $1$ 
and the double edges have 
multiplicity $2$. The denominator at $u$ equals to $\mathcal{W}_\Lambda(u)= u (\omega_2) - w_0^P(\omega_2)$.

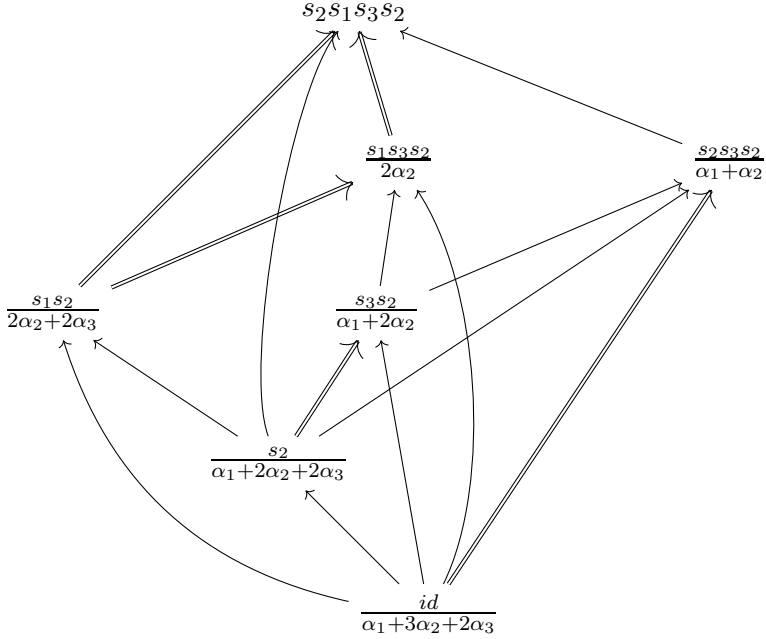
\begin{minipage}{11.5cm}
\begin{tikzpicture}

\node (v0) at (10,8){$s_2 s_1 s_3 s_2$};

\node (v1) at (10.6,6){$\frac{s_1 s_3 s_2}{2 \alpha_2}$};

\node (v2) at (15,6){$\frac{s_2 s_3 s_2}{\alpha_1+\alpha_2}$};

\node (v3) at (6,4){$\frac{s_1 s_2}{2\alpha_2+2 \alpha_3}$};

\node (v4) at (10.3,4){$\frac{s_3 s_2}{\alpha_1+2\alpha_2}$};

\node (v5) at (9,2){$\frac{s_2}{\alpha_1+2\alpha_2+2\alpha_3}$};

\node (v6) at (11,0){$\frac{id}{\alpha_1+3\alpha_2+2\alpha_3}$};

\draw[<-,double distance=0.8pt] (v0)-- node[right]{}(v1);
\draw[<-] (v0)--node[right] {}(v2);
\draw[<-,double distance=0.8pt] (v0)--node[left] {}(v3);

\draw[<-,double distance=0.8pt] (v1)--node[above] {}(v3);
\draw[<-] (v1)--node {}(v4);

\draw[<-] (v2)--node[left] {}(v4);
\draw[<-] (v2)--node {}(v5);
\draw[<-,double distance=0.8pt] (v2)--node[right] {}(v6);

\draw[<-] (v3)--node[right] {}(v5);
\draw[<-,double distance=0.8pt] (v4)--node {}(v5);
\draw[<-] (v5)--node{}(v6);

\draw[bend right,distance=2cm] (v3) to node{}(v6)[<-];
\draw[<-] (v4)--node{}(v6);

\draw[bend right,distance=1cm] (v0) to 
node[below] {${}$} 
(v5) [<-];
\draw[bend left,distance=1.5cm] (v1) to
node[below] {} (v6)[<-];
\end{tikzpicture}
\end{minipage}

Every Schubert 
variety is smooth at $s_2 s_1 s_3s_2$ (the Schubert point), except for $Y(s_2)$ (the Schubert divisor). This may be checked directly e.g. by the smoothness criterion from \Cref{thm:kumar}; we will also recover it from \Cref{thm:genNak} in an example below.

\subsubsection{Example 1: $v=s_3 s_2$, $w=s_2 s_1 s_3s_2$} There are two directed paths from elements $u \ge v$ to $w$ of length $2$, two directed paths of length $1$, and the trivial path, as follows:
\begin{itemize} \item The path $v \to s_1 s_3 s_2 \to w$ contributes with 
$\frac{1}{\alpha_1+2\alpha_2} \times \frac{2}{2 \alpha_2}$;

\item The path $v \to s_2 s_3 s_2 \to w$ contributes with
$\frac{1}{\alpha_1 + 2 \alpha_2}\times  \frac{1}{\alpha_1 + \alpha_2}$;

\item The path $s_1 s_3 s_2 \to w$ contributes with
$\frac{2}{2\alpha_2}$;

\item  The path $s_2 s_3 s_2 \to w$ contributes with
$\frac{1}{\alpha_1+ \alpha_2}$;

\item The trivial path $w=w$ contributes with $1$.
\end{itemize}
In this case \Cref{thm:genNak} (or \Cref{col:NakP}) states that
\[ \begin{split} 1+ \frac{1}{\alpha_1+ \alpha_2} + \frac{2}{2\alpha_2} +\frac{1}{\alpha_1 + 2 \alpha_2}\times  \frac{1}{\alpha_1 + \alpha_2} + \frac{1}{\alpha_1+2\alpha_2} \times \frac{2}{2 \alpha_2}
= \left(1+ \frac{1}{\alpha_2}\right) \left(1+ \frac{1}{\alpha_1+\alpha_2}\right) \/. \end{split}\]

\subsubsection{Example 2: $v=s_2$, $w=s_2s_3 s_2$} In this case $Y(s_2)$ is singular at $w$, therefore we do not expect
the identity from \Cref{thm:genNak} to hold. (However, observe that $Y(s_2)$ is rationally smooth at $w$, since 
$\ell(w) - \ell(v) <3$.) Then:
\begin{itemize}
\item The path $v \to s_3 s_2 \to w$ contributes with 
$\frac{2}{\alpha_2 + 2 \alpha_3} \times \frac{1}{\alpha_2}$;

\item The path $s_2 \to w$ contributes with $\frac{1}{\alpha_2+ 2 \alpha_3}$;

\item The path $s_3 s_2 \to w$ contributes with $\frac{1}{ \alpha_2}$;
\item The trivial path contributes with $1$.
\end{itemize}
The roots $\beta \in S(w/v)$ are $\alpha_2,\alpha_2+ 2 \alpha_3$. Observe that 
$$
1+\frac{1}{\alpha_2}+\frac{1}{\alpha_2+2 \alpha_3}+\frac{1}{\alpha_2}\times \frac{2}{\alpha_2+2 \alpha_3}
\neq
\left(1+\frac{1}{\alpha_2}\right)\left(1+\frac{1}{\alpha_2+2 \alpha_3}\right).
$$
By \Cref{thm:genNak}, this confirms that $Y(s_2)$ is singular at $s_1s_3s_2$.

\subsection{More examples of heaps and excited diagrams}
The purpose of this section is to show two examples of heaps and excited diagrams, 
as utilized in section \ref{sec:heaps}. 

\begin{example}\label{Gr510} In this example we show how the heap from 
\Cref{ex:Heap1} above may be seen as a Grassmannian example associated to a skew shape.
Let $X=\Gr(5,10)$ and $w=s_2 s_1\cdot s_3 s_2 \cdot s_5 s_4 s_3 \cdot 
s_8 s_7 s_6 s_5 s_4 \cdot 
s_9 s_8 s_7 s_6 s_5$, 
$v=  s_4 s_3\cdot  s_5 s_4\cdot  s_6 s_5$.
In terms of partitions, $w$ corresponds to 
$(5,5,3,2,2)$ and $v$ to $(2,2,2)$; cf. 
\Cref{example:Gr512}.
Since the skew shape $wv^{-1} = (5,5,3,2,2)/(2,2,2)$ is a union of straight shapes,
$Y(v)$ is smooth at $w$ and $|\mathcal E_{H(w)} (H(v))|=1$. (As usual, $H(v)$ is indicated in red, and
$H(w) \setminus H(v) = H(w v^{-1})=\{p_1,\ldots, p_{11}\}$.)
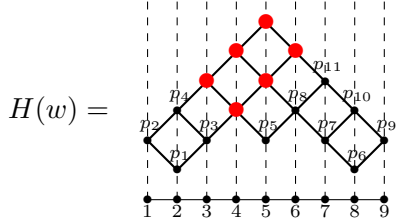
\begin{figure}[h]
\begin{tikzpicture}[x=1.0em,y=1.0em]

\foreach \x in {1,...,9}{
\node at (\x,0) {\tiny $\bullet$};
\node at (\x,-0.4) {\tiny $\x$};
\draw [dashed] (\x,0)--(\x,6.7);
};

\foreach \x/\y/\z in {2/1/1,1/2/2,3/2/3,2/3/4}{
\node at (\x,\y) {\tiny$\bullet$};
\node at (\x+0.1,\y+0.5) {\tiny$p_{\z}$};
};
\draw [thick](1,2)--(2,3)--(3,2)--(2,1)--cycle;

\foreach \x/\y/\z in {5/2/5,8/1/6,7/2/7,6/3/8,7/4/11,8/3/10,9/2/9
}{
\node at (\x,\y) {\tiny $\bullet$};
\node at (\x+0.1,\y+0.5) {\tiny$p_{\z}$};
};

\draw [thick] (5,2)--(7,4);
\draw [thick] (6,3)--(8,1);
\draw [thick] (7,4)--(9,2);
\draw [thick] (7,2)--(8,3);
\draw [thick] (8,1)--(9,2);

\draw (1,0)--(9,0);

\draw [thick](2,3)--(5,6)--(7,4);
\draw [thick](3,2)--(6,5);
\draw [thick](3,4)--(5,2);
\draw [thick](4,5)--(6,3);

\foreach \x/\y in {3/4,4/5,5/6,4/3,5/4,6/5}{
\node[red] at (\x,\y) {\Large$\bullet$};};

\node at (-2,3) {$H(w)=$};

\end{tikzpicture}
\end{figure}

\end{example}

\begin{example}\label{ex:e8} Consider the Lie type $E_8$, with 
$w=s_2\cdot s_4 s_5\cdot s_3 s_4 s_2 \cdot s_1 s_3 s_4 s_5 s_6 s_7 s_8$ and 
$v=s_2 s_4 s_5 s_6 s_7 s_8 $. The element $w$ is $\varpi_8$-minuscule.
The heap for the skew shape $wv^{-1}$ is illustrated in the figure below, together
with the set $\mathcal E_{H(w)}(H(v))$. By \Cref{lemma:eqmult},
%the multiplicity of $Y(v)$ at $w$ is 
$\mult_{w}(Y(v))=5$. 
\begin{figure}[h!]
\begin{tikzpicture}[x=1.0em,y=0.7em]%%%%% 
\node at (-1,3) {$H(w)=$};

\foreach \x in {1,...,7}{
\node at (\x,9-\x) {\tiny $\bullet$};};

\draw [thick](1,8)--(7,2);

%%%%%%%%%%%%%%%%%%%%%%%%%%%%%%%
\foreach \x/\y in {3/8,4/9,5/10,6/11,7/12,8/13}{
\node at (9-\x,\x+0.5) {\tiny $p_{\y}$};};

\foreach \x/\y in {4/1,5/2,6/1,5/0,4.5/3,4.5/-1}{
\node at (\x,\y) {\tiny $\bullet$};};

\foreach \x/\y/\z in {4.5/3/6,4.5/-1/1}{
\node at (\x,\y+0.7) {\tiny $p_{\z}$};};
\foreach \x/\y/\z in {5/2/5,5/0/2}{
\node at (\x,\y+0.5) {\tiny $p_{\z}$};};
\node at (7,2+0.5) {\tiny $p_7$};
\node at (6,1+0.5) {\tiny $p_4$};
\node at (4,1+0.5) {\tiny $p_3$};
%%%%%%%%%%%%%%%%%%%%%%%%%%%%%%%

\draw [thick](5,4)--(4.5,3)--(5,2);

\draw [thick](5,0)--(4.5,-1);

\draw [thick](4,1)--(5,2)--(6,1)--(5,0)--cycle;
\draw [thick](6,3)--(5,2);
\draw [thick](6,1)--(7,2);

\draw [thick](1,8)--(7,2);

\draw [thick](1,-2)--(7,-2);
\foreach \x in {1,...,7}{
\node at (\x,-2) {\tiny $\bullet$};};

\node at (4.5,-3) {\tiny $\bullet$};
\draw [thick](4.5,-3)--(5,-2);

\foreach \x in {1,...,7}{
\draw [dashed] (\x,-2)--(\x,9);
};
\foreach \x in {3,...,8}{
\node at (9-\x,-2.5) {\tiny $\x$};};
\node at (7,-2.5) {\tiny $1$};
\node at (4.5,-3.5) {\tiny $2$};
\draw [dashed] (4.5,-3)--(4.5,4);

\end{tikzpicture}

\begin{tikzpicture}[x=1.0em,y=0.7em]%%%%%%%% EYD1

\foreach \x in {1,...,7}{
\node at (\x,9-\x) {\tiny $\bullet$};};

\draw [thick](1,8)--(7,2);

%%%%%%%%%%%%%%%%%%%%%%%%%%%%%%%
\foreach \x/\y in {3/8,4/9,5/10,6/11,7/12,8/13}{
\node at (9-\x,\x+0.5) {\tiny $p_{\y}$};};

\foreach \x/\y in {4/1,5/2,6/1,5/0,4.5/3,4.5/-1}{
\node at (\x,\y) {\tiny $\bullet$};};

\foreach \x/\y/\z in {4.5/3/6,4.5/-1/1}{
\node at (\x,\y+0.7) {\tiny $p_{\z}$};};
\foreach \x/\y/\z in {5/2/5,5/0/2}{
\node at (\x,\y+0.5) {\tiny $p_{\z}$};};
\node at (7,2+0.5) {\tiny $p_7$};
\node at (6,1+0.5) {\tiny $p_4$};
\node at (4,1+0.5) {\tiny $p_3$};

\node[red] at (4,8) {$D_1$};

\draw [thick](5,4)--(4.5,3)--(5,2);
\draw [thick](5,0)--(4.5,-1);

\draw [thick](4,1)--(5,2)--(6,1)--(5,0)--cycle;
\draw [thick](6,3)--(5,2);
\draw [thick](6,1)--(7,2);

\draw [thick](1,8)--(7,2);

\draw [thick](1,-2)--(7,-2);
\foreach \x in {1,...,7}{
\node at (\x,-2) {\tiny $\bullet$};};

\node at (4.5,-3) {\tiny $\bullet$};
\draw [thick](4.5,-3)--(5,-2);

\foreach \x in {1,...,7}{
\draw [dashed] (\x,-2)--(\x,9);
};
\foreach \x in {3,...,8}{
\node at (9-\x,-2.5) {\tiny $\x$};};
\node at (7,-2.5) {\tiny $1$};
\node at (4.5,-3.5) {\tiny $2$};
\draw [dashed] (4.5,-3)--(4.5,4);

\foreach \x in {4,...,8}{
\node[red] at (9-\x,\x) {\large$\bullet$};};
\node[red] at (4.5,3) {\large$\bullet$};

\end{tikzpicture}%%%%%%%%%%%%%%%%%
\begin{tikzpicture}[x=1.0em,y=0.7em]%%%%% EYD2

\foreach \x in {1,...,7}{
\node at (\x,9-\x) {\tiny $\bullet$};};

\draw [thick](1,8)--(7,2);

%%%%%%%%%%%%%%%%%%%%%%%%%%%%%%%
\foreach \x/\y in {3/8,4/9,5/10,6/11,7/12,8/13}{
\node at (9-\x,\x+0.5) {\tiny $p_{\y}$};};

\foreach \x/\y in {4/1,5/2,6/1,5/0,4.5/3,4.5/-1}{
\node at (\x,\y) {\tiny $\bullet$};};

\foreach \x/\y/\z in {4.5/3/6,4.5/-1/1}{
\node at (\x,\y+0.7) {\tiny $p_{\z}$};};
\foreach \x/\y/\z in {5/2/5,5/0/2}{
\node at (\x,\y+0.5) {\tiny $p_{\z}$};};
\node at (7,2+0.5) {\tiny $p_7$};
\node at (6,1+0.5) {\tiny $p_4$};
\node at (4,1+0.5) {\tiny $p_3$};

\node[red] at (4,8) {$D_2$};
%%%%%%%%%%%%%%%%%%%%%%%%%%%%%%%

\draw [thick](5,4)--(4.5,3)--(5,2);

\draw [thick](5,0)--(4.5,-1);

\draw [thick](4,1)--(5,2)--(6,1)--(5,0)--cycle;
\draw [thick](6,3)--(5,2);
\draw [thick](6,1)--(7,2);

\draw [thick](1,8)--(7,2);

\draw [thick](1,-2)--(7,-2);
\foreach \x in {1,...,7}{
\node at (\x,-2) {\tiny $\bullet$};};

\node at (4.5,-3) {\tiny $\bullet$};
\draw [thick](4.5,-3)--(5,-2);

\foreach \x in {1,...,7}{
\draw [dashed] (\x,-2)--(\x,9);
};
\foreach \x in {3,...,8}{
\node at (9-\x,-2.5) {\tiny $\x$};};
\node at (7,-2.5) {\tiny $1$};
\node at (4.5,-3.5) {\tiny $2$};
\draw [dashed] (4.5,-3)--(4.5,4);

\foreach \x in {4,...,8}{
\node[red] at (9-\x,\x) {\large$\bullet$};};
\node[red] at (4.5,-1) {\large$\bullet$};

\end{tikzpicture}%%%%%%%%%%%%%%%%%
\begin{tikzpicture}[x=1.0em,y=0.7em]%%%%% EYD 3

\foreach \x in {1,...,7}{
\node at (\x,9-\x) {\tiny $\bullet$};};

\draw [thick](1,8)--(7,2);
%%%%%%%%%%%%%%%%%%%%%%%%%%%%%%%
\foreach \x/\y in {3/8,4/9,5/10,6/11,7/12,8/13}{
\node at (9-\x,\x+0.5) {\tiny $p_{\y}$};};

\foreach \x/\y in {4/1,5/2,6/1,5/0,4.5/3,4.5/-1}{
\node at (\x,\y) {\tiny $\bullet$};};

\foreach \x/\y/\z in {4.5/3/6,4.5/-1/1}{
\node at (\x,\y+0.7) {\tiny $p_{\z}$};};
\foreach \x/\y/\z in {5/2/5,5/0/2}{
\node at (\x,\y+0.5) {\tiny $p_{\z}$};};
\node at (7,2+0.5) {\tiny $p_7$};
\node at (6,1+0.5) {\tiny $p_4$};
\node at (4,1+0.5) {\tiny $p_3$};

\node[red] at (4,8) {$D_3$};
%%%%%%%%%%%%%%%%%%%%%%%%%%%%%%%

\draw [thick](5,4)--(4.5,3)--(5,2);

\draw [thick](5,0)--(4.5,-1);

\draw [thick](4,1)--(5,2)--(6,1)--(5,0)--cycle;
\draw [thick](6,3)--(5,2);
\draw [thick](6,1)--(7,2);

\draw [thick](1,8)--(7,2);

\draw [thick](1,-2)--(7,-2);
\foreach \x in {1,...,7}{
\node at (\x,-2) {\tiny $\bullet$};};

\node at (4.5,-3) {\tiny $\bullet$};
\draw [thick](4.5,-3)--(5,-2);

\foreach \x in {1,...,7}{
\draw [dashed] (\x,-2)--(\x,9);
};
\foreach \x in {3,...,8}{
\node at (9-\x,-2.5) {\tiny $\x$};};
\node at (7,-2.5) {\tiny $1$};
\node at (4.5,-3.5) {\tiny $2$};
\draw [dashed] (4.5,-3)--(4.5,4);

\foreach \x in {5,...,8}{
\node[red] at (9-\x,\x) {\large$\bullet$};};
\node[red] at (4.5,-1) {\large$\bullet$};
\node[red] at (5,2) {\large$\bullet$};

\end{tikzpicture}
\begin{tikzpicture}[x=1.0em,y=0.7em]%%%%% EYD 4

\foreach \x in {1,...,7}{
\node at (\x,9-\x) {\tiny $\bullet$};};

\draw [thick](1,8)--(7,2);
%%%%%%%%%%%%%%%%%%%%%%%%%%%%%%%
\foreach \x/\y in {3/8,4/9,5/10,6/11,7/12,8/13}{
\node at (9-\x,\x+0.5) {\tiny $p_{\y}$};};

\foreach \x/\y in {4/1,5/2,6/1,5/0,4.5/3,4.5/-1}{
\node at (\x,\y) {\tiny $\bullet$};};

\foreach \x/\y/\z in {4.5/3/6,4.5/-1/1}{
\node at (\x,\y+0.7) {\tiny $p_{\z}$};};
\foreach \x/\y/\z in {5/2/5,5/0/2}{
\node at (\x,\y+0.5) {\tiny $p_{\z}$};};
\node at (7,2+0.5) {\tiny $p_7$};
\node at (6,1+0.5) {\tiny $p_4$};
\node at (4,1+0.5) {\tiny $p_3$};

\node[red] at (4,8) {$D_4$};

%%%%%%%%%%%%%%%%%%%%%%%%%%%%%%%

\draw [thick](5,4)--(4.5,3)--(5,2);

\draw [thick](5,0)--(4.5,-1);

\draw [thick](4,1)--(5,2)--(6,1)--(5,0)--cycle;
\draw [thick](6,3)--(5,2);
\draw [thick](6,1)--(7,2);

\draw [thick](1,8)--(7,2);

\draw [thick](1,-2)--(7,-2);
\foreach \x in {1,...,7}{
\node at (\x,-2) {\tiny $\bullet$};};

\node at (4.5,-3) {\tiny $\bullet$};
\draw [thick](4.5,-3)--(5,-2);

\foreach \x in {1,...,7}{
\draw [dashed] (\x,-2)--(\x,9);
};
\foreach \x in {3,...,8}{
\node at (9-\x,-2.5) {\tiny $\x$};};
\node at (7,-2.5) {\tiny $1$};
\node at (4.5,-3.5) {\tiny $2$};
\draw [dashed] (4.5,-3)--(4.5,4);

\foreach \x in {5,...,8}{
\node[red] at (9-\x,\x) {\large$\bullet$};};
\node[red] at (4.5,-1) {\large$\bullet$};
\node[red] at (5,0) {\large$\bullet$};

\end{tikzpicture}%%%%%%%%%%%%%%%%%
\begin{tikzpicture}[x=1.0em,y=0.7em]%%%%%% EYD 5

\foreach \x in {1,...,7}{
\node at (\x,9-\x) {\tiny $\bullet$};};

\draw [thick](1,8)--(7,2);
%%%%%%%%%%%%%%%%%%%%%%%%%%%%%%%
\foreach \x/\y in {3/8,4/9,5/10,6/11,7/12,8/13}{
\node at (9-\x,\x+0.5) {\tiny $p_{\y}$};};

\foreach \x/\y in {4/1,5/2,6/1,5/0,4.5/3,4.5/-1}{
\node at (\x,\y) {\tiny $\bullet$};};

\foreach \x/\y/\z in {4.5/3/6,4.5/-1/1}{
\node at (\x,\y+0.7) {\tiny $p_{\z}$};};
\foreach \x/\y/\z in {5/2/5,5/0/2}{
\node at (\x,\y+0.5) {\tiny $p_{\z}$};};
\node at (7,2+0.5) {\tiny $p_7$};
\node at (6,1+0.5) {\tiny $p_4$};
\node at (4,1+0.5) {\tiny $p_3$};

\node[red] at (4,8) {$D_5$};

%%%%%%%%%%%%%%%%%%%%%%%%%%%%%%%

\draw [thick](5,4)--(4.5,3)--(5,2);

\draw [thick](5,0)--(4.5,-1);

\draw [thick](4,1)--(5,2)--(6,1)--(5,0)--cycle;
\draw [thick](6,3)--(5,2);
\draw [thick](6,1)--(7,2);

\draw [thick](1,8)--(7,2);

\draw [thick](1,-2)--(7,-2);
\foreach \x in {1,...,7}{
\node at (\x,-2) {\tiny $\bullet$};};

\node at (4.5,-3) {\tiny $\bullet$};
\draw [thick](4.5,-3)--(5,-2);

\foreach \x in {1,...,7}{
\draw [dashed] (\x,-2)--(\x,9);
};
\foreach \x in {3,...,8}{
\node at (9-\x,-2.5) {\tiny $\x$};};
\node at (7,-2.5) {\tiny $1$};
\node at (4.5,-3.5) {\tiny $2$};
\draw [dashed] (4.5,-3)--(4.5,4);

\foreach \x in {6,...,8}{
\node[red] at (9-\x,\x) {\large$\bullet$};};
\node[red] at (4.5,-1) {\large$\bullet$};
\node[red] at (5,0) {\large$\bullet$};
\node[red] at (4,1) {\large$\bullet$};

\end{tikzpicture}%%%%%%%%%%%%%%%%%

$\mathcal E_{H(w)}(H(v))
=\{
{\color{red}
D_1\vartriangleright  D_2\vartriangleright  D_3
\vartriangleright  D_4\vartriangleright  D_5}
\}$
%\caption{The $E_8$ heap $H(w)$ and excited diagrams in $\mathcal E_{H(w)}(H(v))$}
% for \\$w=$
%$s_2\cdot s_4 s_5\cdot s_3 s_4 s_2 \cdot s_1 s_3 s_4 s_5 s_6 s_7 s_8$,
%$v=s_2 s_4 s_5 s_6 s_7 s_8$.}
\label{ex:e8eyd}
\end{figure}
\end{example}
%\newpage

\bibliographystyle{halpha}
\bibliography{hook.bib}

\end{document}